\documentclass[times,12pt]{article}
\usepackage[english]{babel}
\usepackage{amsmath,amsfonts,amstext}
\usepackage{graphicx}
\usepackage[latin1]{inputenc}
\usepackage{psfrag}
\usepackage{mathenv}
\usepackage{multicol}
\usepackage{setspace}
\usepackage{color}
\usepackage{geometry}
\newtheorem{Thm}{Theorem}
\newtheorem{Prop}{Proposition}

\newtheorem{Cor}{Corollary}
\newtheorem{Ass}{Assumption}
\newtheorem{Lem}{Lemma}
\newtheorem{Def}{Definition}

\newenvironment{proof}[1][Proof]{\textbf{#1.} }{\ \rule{0.5em}{0.5em}}

\newcommand\indep{\protect\mathpalette{\protect\independenT}{\perp}}
\def\independenT#1#2{\mathrel{\rlap{$#1#2$}\mkern2mu{#1#2}}}

\newcommand{\eps}{\varepsilon}
\newcommand{\bs}{\backslash}

\newcommand{\G}{\mathcal{G}} 
\newcommand{\gzu}{\G_{0-1}}
\newcommand{\Gfull}{\G_{\textit{full}}}
\newcommand{\Gmin}{\tilde{\G}}\newcommand{\Gq}{\G^{(q)}}
\newcommand{\gq}{\G^{(q)}}
\newcommand{\gun}{\G^{(1)}}
\newcommand{\gunhat}{\hat{\G}^{(1)}}

\newcommand{\p}{P}
\newcommand{\n}{N}
\newcommand{\R}{\mathbb{R}}
\newcommand{\proba}{\mathbb{P}}
\newcommand{\pj}{\p_{j}}
\newcommand{\Pmi}{\p_{i}}
\newcommand{\NpaxitGmin}{N_{\text{pa}}(X^i_t, \Gmin)}
\newcommand{\NpaMaxGmin}{N_{\text{pa}}^{\text{Max}}(\Gmin)}
\newcommand{\ie}{\textit{i.e.} \rm}

\title{\huge{Inferring dynamic genetic networks\\
 with low order independencies}
\author{Sophie L\`ebre$^{\star}$\\
 \texttt{s.lebre@imperial.ac.uk}\\
}
}

\date{}

\begin{document}

\maketitle

\vspace{-0.5cm}

\begin{center}
Universit\'e d'Evry-Val-d'Essone,
CNRS UMR 8071,
INRA 1152,\\
Laboratoire Statistique et G\'enome\\
523 place des Terrasses,
91000 Evry,
France.\\
\vspace{0.3cm}
$^{\star}$\small{Current address:  Centre for Bioinformatics,
Division of Molecular Biosciences,}\\
\small{Imperial College London,
South Kensington Campus,
SW7 2AZ London, UK.}
\end{center}


  \begin{center}
    \textbf{Abstract}
  \end{center}
  In  this paper,  we propose  a novel  inference method  for  dynamic genetic
  networks which makes it possible to  deal with a number of time measurements
  $n$ much smaller than the number of genes $p$.  The approach is based on the
  concept of low order conditional dependence graph which we extend here to the
  case of Dynamic Bayesian Networks.  Most of our results are based
  on  the theory  of graphical  models  associated with  Directed  Acyclic
  Graphs (DAGs).  In this way, we define a DAG $\Gmin$ which describes exactly
  the  \textit{full order  conditional  dependencies} given  the  past of  the
  process.  Then, to cope with the large $p$ and small $n$ estimation case, we
  propose  to approximate  DAG $\Gmin$  by considering  low  order conditional
  independencies.  We introduce  partial $q^{th}$ order conditional dependence
  DAGs  and analyze their  probabilistic properties.   In general,  DAGs $\gq$
  differ from $\Gmin$  but still reflect relevant dependence  facts for sparse
  networks such as genetic networks.   By using this approximation, we set out
  a non-Bayesian  inference method and  demonstrate the effectiveness  of this
  approach on both simulated and  real data analysis.  The inference procedure
  is implemented in the R package 'G1DBN' which is available from the CRAN archive. 



\section*{Introduction}\label{sec:intro}
The development of microarray technology allows to simultaneously measure
the expression levels of many genes at a precise time point. 
Thus it has become possible to observe gene expression levels across
a  whole  process  such as the  cell  cycle  or response  to  radiation  or different
treatments.  The objective is now to recover gene regulation phenomena from this data. 
We are looking for simple relationships such as ``gene $i$ activates gene~$j$''. 
But we also want to capture more complex scenarios such as auto-regulations,
feed-forward loops, multi-component loops...  as described by Lee et al. 
\cite{Lee02} in the case of the transcriptional regulatory network of the yeast \textit{Saccharomyces cerevisiae}.

To such an aim, we both need to accurately take into account
temporal dependencies and to deal with the dimension of the problem when
the number $p$ of observed genes is much higher than the number $n$ of observation
time points. 
Moreover we know that 
most  of the  genes whose expression has been monitored using microarrays are  not taking  part  in the temporal evolution of  the system. So we  want to determine the few  `active' genes that are involved in the regulatory machinery,
as well as  the relationships between them. 
In short, we want to infer a network representing the dependence relationships
which govern a system composed  of several agents
from the observation  of their activity across short time series.

\paragraph{Static  Modelling}
Such  gene networks  were first described using static modelling
and mainly non oriented networks. 
One of the first tools used to describe
interactions   between    genes   is   the    \textit{relevance   network}  \cite{Butte00}  or
\textit{correlation network}~\cite{Streuer03}.  Better known as the
\textit{covariance graph} \cite{Cox96} in  graphical models theory, this
undirected graph describes the pair-wise correlation between genes. 
Its  topology is derived from 
the covariance matrix between the gene
expression levels; an undirected edge  is drawn between two variables whenever
they are correlated. However, the correlation between two
variables may be caused by linkage with other variables. This creates spurious
edges due to indirect dependence relationships.

 Consequently, there has been great interest in the \textit{concentration graph} \cite{Lauritzen96}, also called
the \textit{covariance selection} model,
which  describes the  \textit{conditional}   dependence  structure  between gene
expression using Graphical Gaussian Models (GGMs).
 Let $Y\!=\!(Y^i)_{1\leq i \leq p}$ be a multivariate Gaussian vector representing
 the expression levels of $p$ genes.  An undirected edge is drawn
between two variables $Y^i$ and $Y^j$ whenever they are conditionally dependent given the
remaining variables  (See  Figure  \ref{fig:motifDBN}B). 
The standard theory of estimation in GGMs~\cite{Lauritzen96,Whittaker90}
can be exploited only when the number of measurements $n$ is much higher
than the number of variables $ p$. This ensures that the sample covariance matrix is positive definite with
probability one. However, in most microarray gene expression datasets,
we have to cope with the opposite situation $(n<<p)$. 
Thus, the growing interest in ``small $n$, large $p$'' furthered the development of numerous alternatives (Sch\"afer and
Strimmer \cite{Strimmer05a,Strimmer05b} , Waddell and Kishino \cite{Waddell00a,Waddell00b},  Toh  and   Horimoto  \cite{Toh02a,Toh02b},  Wu  et
al. \cite {Wu03}, Wang et al. \cite{Wang03}). 
Even  though   concentration  graphs  allow  to  point   out  some  dependence
relationships between genes, they do  not offer an accurate description of the
interactions. Firstly,  no direction is  given to the  interactions. Secondly,
some  motifs containing  cycles as in Figure \ref{fig:motifDBN}A cannot be  properly  represented.

Contrary  to   the  previous   undirected  graphs,  Bayesian   networks  (BNs)
\cite{Friedman00}
 model directed relationships.  Based on a probabilistic measure, a BN representation of a
model is defined by a Directed Acyclic Graph (DAG) and the
set of conditional probability  distributions  of  each  variable given  its
parents  in  the DAG  \cite{Pearl88}. The theory of graphical models
\cite{Whittaker90,Edwards95,Lauritzen96} then allows to derive conditional
independencies from this DAG.  However, the acyclicity
constraint in static BNs is a serious restriction given the expected structure of
genetic networks.   

\begin{figure}
\begin{picture}(10,10)
\put(20,-10){A \hspace{6.2cm} B}
\put(20,-63){C}
\end{picture}

\begin{center}
\includegraphics[width=11cm]{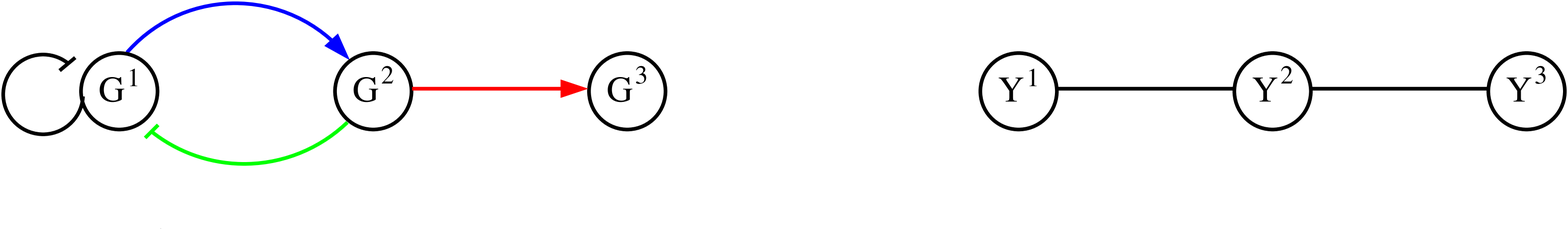}\\
\includegraphics[width=10cm]{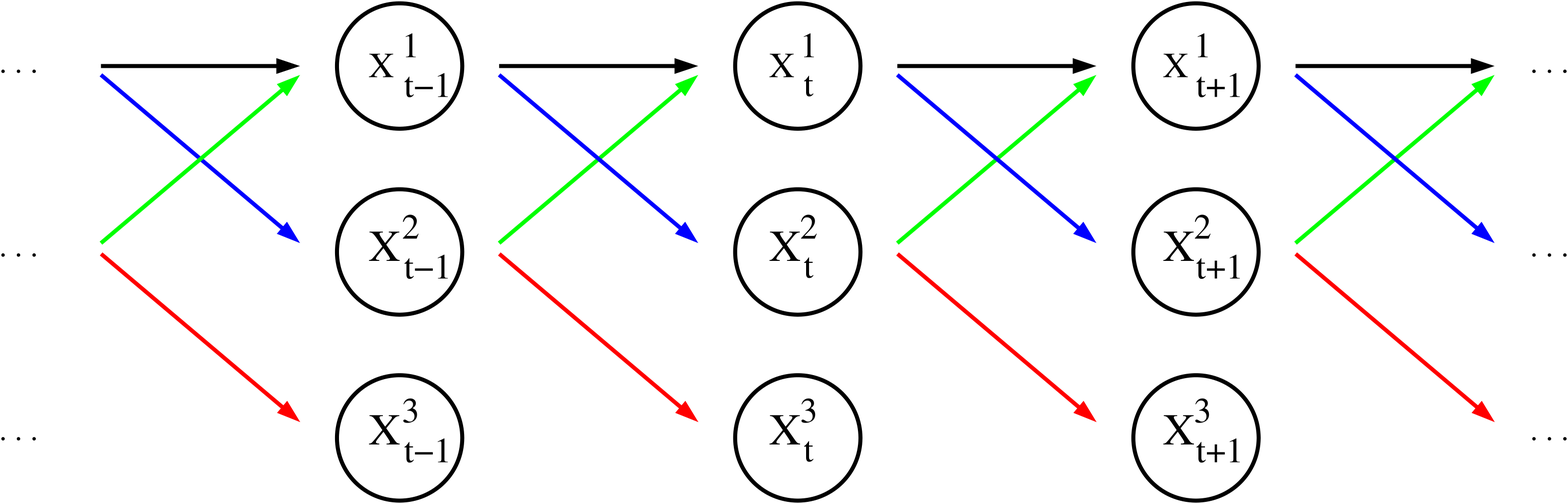}
\end{center}
\vspace{-0.5cm}
\caption{(A) A biological regulation motif. (B) The concentration graph corresponding to the motif A.  For all $i \geq 3$, $Y^i$ is a Gaussian variable representing the expression  level of  gene~$G^i$.  Some cycles  cannot be  represented  on the
concentration graph. (C) Dynamic  network  equivalent  to   the  regulation  motif A. Each vertex $X^i_t$ represents the expression level of gene $G^i$ at  time $t$. This graph is acyclic and allows to define a Bayesian network. }
\label{fig:motifDBN}
\end{figure}

\paragraph{Dynamic Bayesian networks}
This limitation can be overcome by employing  Dynamic Bayesian networks (DBNs) introduced for the
analysis of gene  expression time series by Friedman  et al. \cite{Friedman98}
and Murphy and Mian \cite{Murphy99}. In DBNs, a gene is no longer represented
by a single vertex but by as many vertices as time points in the
experiment. A dynamic network (Figure \ref{fig:motifDBN}C)  can then be obtained
by   \textit{unfolding  in   time}  the   initial  cyclic   motif   in  Figure
\ref{fig:motifDBN}A. The direction according to time guarantees 
the acyclicity of this dynamic network and consequently allows to define a Bayesian
network. 
The nature of the relationships  (positive/negative) 
does not appear in this DAG but is derived from  estimates of the model parameters. 

The very high  number $p$ of genes simultaneously  observed raises a dimension
problem.   Moreover, a  large majority  of  time series  gene expression  data
contain no or very few repeated  measurements of the expression level of the
same  gene   at  a  given  time.   Hence,  we  assume  that   the  process  is
\textit{homogeneous} across time. This means that the system
is considered to be governed by the same rules during the whole experiment.
Consequently, the temporal dependencies are homogeneous: any edge
is present or absent during the whole process. 
 This is a strong assumption which is not necessarily satisfied. Nevertheless,
 this condition is necessary to carry out estimation unless we have several measurements of each gene expression at each time point. 

Up to now, various DBN representations based on different probabilistic
models    have   been    proposed    (discrete   models    \cite{Ong02,Zou05},
multivariate auto-regressive process \cite{OpgenStrimmer07}, State
Space  or Hidden  Markov Models  \cite{Perrin03,Wu04,Rangel04,Beal05},
nonparametric additive regression model \cite{Imoto02,Imoto03,Kim04,Sugimoto04}).  See   also    Kim   et
al. \cite{Kim03} for a review of such models. 
Faced with so  much diversity, we introduce in this paper sufficient  conditions  for a model to admit a DBN representation
and  we set out  a concrete   interpretation in  terms of  dependencies between
variables by using the theory of graphical models for DAGs.

Our DBN representation is based on a DAG $\Gmin$
(e.g. like the DAG of Fig. \ref{fig:motifDBN}C) which describes exactly the full order conditional dependencies
given all the remaining \textit{past} variables (See Section~\ref{sec:DBN}). 
This approach extends the principle of the concentration graph
showing conditional independencies to the dynamic case. 

\paragraph{Dimension reduction}
Even under  the assumption of homogeneity,  which enables to  use the pairs of successive time
point gene expression  as repeated measurements,
we  
have to deal with the ``curse of dimensionality'' when inferring the 
structure of DAG $\Gmin$.  The difficulty lies in coping with the large $p$ and small $n$
estimation case.  Several inference methods have been proposed for the
estimation of the  topology of the DAG defining the various DBNs  quoted above. To name a few,
Murphy \cite{Murphy01} implemented several Bayesian structure
learning procedures for 
dynamic models in the 
Matlab package BNT (Bayes
Net Toolbox);
Ong et al.  \cite{Ong02} reduce the
dimension of the problem by considering prior knowledge; Perrin et al. \cite{Perrin03}
use an extension of the linear regression;
Wu et al.  \cite{Wu04} use factor analysis and Beal et
al.   \cite{Beal05} develop  a  variational Bayesian  method;  Zou and  Conzen
\cite{Zou05} limit  potential regulators to  the genes with either  earlier or
simultaneous expression changes and estimate the transcription time lag; Opgen-Rhein and
Strimmer \cite{OpgenStrimmer07} proposed a model
selection procedure  based on an  analytic shrinkage approach.
However, a powerful approach based on the consideration of zero- and first-order conditional independencies to   model concentration graphs  has  gained  attention. 
When $n<<p$, Wille et al.  \cite{Wille04,Wille06} propose to approximate the
concentration graph by the graph $\gzu$ describing zero- and first-order conditional
independence. An  edge between the variables  $Y^i$ and $Y^j$ is  drawn in the
graph $\G_{0-1}$  if and only  if, zero- and first-order  correlations between
these two variables both differ from zero, that is, if 
\begin{equation}
r(Y^i,Y^j)\not=0 \ \ \text{ and } \ \ \forall
k \in \{1,...,p\} \bs \{i,j\}, \ r(Y^i,Y^j \vert Y^k) \not=0,
\end{equation}
\noindent where $r(Y^i,Y^j \vert Y^k)$ is the partial correlation between 
$Y^i$ and $Y^j$ given $Y^k$. 
Hence, whenever the correlation between two variables $Y^i$ and $Y^j$ can be
entirely explained by the effect of some variable $Y^k$, no edge is drawn between them.  

This procedure allows a drastic dimension reduction: by   using   first   order   conditional
correlations, estimation can be carried out accurately even with a small
number of observations. 
Even if the graph 
of zero- and first-order conditional independence differs from
the concentration graph 
in general, it still reflects some measure of conditional independence.  Wille
et  al.   show  through  simulations  that  the graph  $\gzu$  offers  a  good
approximation of sparse concentration  graphs and demonstrate that both graphs
coincide exactly if the  concentration graph is a forest (\cite{Wille06},
Corollary  1).   This  approach  has   also  been  used  by  Magwene  and  Kim
\cite{Magwene04}  and de  la  Fuente et  al.   \cite{Fuente04} for  estimating
undirected  gene  networks from  microarray  gene  expression  of the  yeast
\textit{Saccharomyces cerevisiae}.  Castelo and Roverato  \cite{Roverato06}  
investigate such  undirected $q^{th}$ order partial  independence graphs for
$q \geq 1$ and present a thorough analysis of their properties.  
In this paper, 
we extend  this approach by
defining $q^{th}$ order order conditional dependence DAGs $\gq$ for DBN representations. 
Then, by basing our results on these low order conditional dependence DAGs, we propose a novel inference method
for dynamic genetic networks which makes it possible to deal with the ``small $n$, large $p$'' problem.

The remainder of the paper  is organized as follows. In Section \ref{sec:DBN},
we provide sufficient  conditions for a DBN modelling  of time series describing
temporal dependencies.  In particular, we show the existence of a minimal DAG $\Gmin$
which  allows such  a  DBN representation.   To  reduce the  dimension of  the
estimation of the topology of $\Gmin$,
we propose to approximate $\Gmin$ by $q^{th}$ order conditional
dependence DAGs $\gq$ and analyze their probabilistic properties in Section~\ref{sec:Gq}. 
From conditions on the topology of $\Gmin$ and the faithfulness assumption, we
establish inclusion relationships between both DAGs $\Gmin$ and $\gq$. 
In Section~\ref{sec:estim}, we exploit our results on 
DAGs $\gq$ 

Finally, validation  is obtained  on both simulated  and real data  in Section
\ref{sec:valid}.  We use our inference procedure for the analysis of two 
microarray time course data sets: the Spellman's yeast cell cycle data
\cite{Spellman98} and the diurnal cycle data on the starch metabolism of
\textit{Arabidopsis Thaliana} collected by Smith et al. \cite{Smith04}. 

\begin{table}
\caption{Notations}
\label{tab:specGmin}
\vspace{0.1cm}
\begin{tabular}{|llll|}
\hline
$\p$\!\!\!\!\!\!\!\!&=&\!\!\!\!$\{ 1 \leq i \leq p \}$ &set of the observed genes,\\
 $\Pmi$\!\!\!\!\!\!\!\!&=&\!\!\!\!$ p \bs \{i\}$&set of the observed genes except gene $i$,\\
$\n$\!\!\!\!\!\!\!\!&=&\!\!\!\!$\{ 1 \leq t \leq n \}$ & set of observation times,\\

$X$\!\!\!\!\!\!\!\!&=&\!\!\!\!$\{X^i_t; i \in \p, t \in \n \}$ \!\!\!\!& stochastic process (gene expression 
time series),\!\!\\

$\G $\!\!\!\!\!\!\!\!&=&\!\!\!\!$(X, E(\G))$& a DAG whose vertices are  defined by $X$ and \\
&&&edges by $ E(\G) \subseteq X \times X$,\\
$\Gmin$\!\!\!\!\!\!\!\!&&&  the  ``true''  DAG  describing  the set of\\
&&& full  order  conditional
dependencies, \\ 

$\gq$\!\!\!\!\!\!\!\!&&& $q^{th}$ order conditional dependence DAG, \\ 
\hline
\end{tabular}
\end{table}

\section{A minimal DBN  representation}\label{sec:DBN}

Let $\p\!=\!\{ 1 \leq i \leq p \}$ describe the set of observed genes
and $\n\!=\!\{~1~\leq~t~\leq~n~\}$ the set of observation times. \!\!\!\! In this
paper, \!\!\! we consider a discrete-time stochastic process $X=\{X^i_t; i \in \p, t \in \n \}$
taking real values  and assume the joint probability  distribution $\proba$ of
the process $X$ has density $f$ with respect to Lebesgue measure on $\R^{p \times n}$. 
We denote by $X_t =\{X^i_t; i \in \p \}$ the set of the $p$ random variables observed at time $t$ and
$X_{1:t} = \{X^i_s; i \in \p, s \leq t \}$ the set of the random variables observed before time $t$.

The main result of this section is set out in Proposition \ref{prop:Gmin}; 
we show that process $X$ admits a  DBN representation according to a minimal DAG $\Gmin$
whose edges describe exactly the set of direct dependencies between successive
variables $X^j_{t-1}, X^i_t$ 
given the past of the process. For an illustration, the minimal DAG $\Gmin$ is given in the
 case of an AR(1) model in Subsection \ref{ssec:ar1}.  
Most of our results are derived from the theory of graphical models associated with DAGs~\cite{Lauritzen96}. 
Note that, even though  we need to consider a  homogeneous DBN for the  inference of gene
interaction  networks,  the  theoretical results  introduced in Sections  \ref{sec:DBN}  and
\ref{sec:Gq} are valid without assuming homogeneity across time.

\subsection{Background}

\paragraph{Theory of graphical models associated with DAGs}
\ Let $\G=(X, E(\G))$ be a DAG whose vertices are the variables $X= \{X^i_t; i
\in \p, t \in \n\}$ and whose set of 
edges $ E(\G) $ is a subset of $X \times X$. We quickly recall here elements of the theory of
graphical models associated with DAGs  \cite{Lauritzen96}. A characterization of a Bayesian
Network (BN) representation for a process $X$ is given in Proposition~\ref{Prop:BNrep}.
\begin{Def}\label{Def:parents} \textbf{(Parents, Lauritzen \cite{Lauritzen96})} 
 The \textit{parents} of a vertex $X^i_t$ in $\G$, denoted by $\text{pa}(X^i_t, \G)$,
are the variables having an edge pointing towards the vertex $X^i_t$ in $\G$, 
$$\text{pa}(X^i_t, \G):=\{X^j_{s}  \text{ such that  } (X^j_{s}, X^i_t) \in  E(\G); j
\in \p, s \in \n\}.$$
\end{Def}
\begin{Prop}\label{Prop:BNrep}\textbf{(BN representation, Pearl \cite{Pearl88})}
The probability  distribution $\proba$  of process  $X$ admits  a Bayesian
Network (BN) representation according to DAG $\G$ whenever its density~$f$
factorizes as a product of the conditional density of
each variable $X_i^t$ given its parents in~$\G$,
\vspace{-0.2cm}
$$f(X) = \prod_{i \in \p} \prod_{t \in \n} f(X^i_t\vert \text{pa}(X^i_t,\G)).$$
\end{Prop}
Throughout this paper, a central notion is that of conditional independence of
random variables. 
 Two random variables $U$ and $V$ are \textit{conditionally independent}
 given a third variable $W$  (and we write $U~\indep~V~\vert~W$)
 if they are independent in the joint probability distribution $\proba_{U,V,W}$ of the three random variables $(U,V,W)$. 
In other words, $U$ and $V$ are conditionally independent
 given $W$ if for any possible value $w$ of $W$, variables $U$ and~$V$ are 
independent given the variable $W = w$.
This result generalizes to disjoint sets of variables.
Such   conditional  independence   relationships  can   be  obtained   from   a  BN
representation by using graphical theory associated with DAGs,
which is essentially based  on the \textit{directed global Markov property}
recalled in Proposition \ref{prop:gAnc}.

\begin{figure}[t]
\begin{picture}(10,10)
\put(40,-5){A}
\put(40,-115){ B}
\end{picture}
\begin{center}
\vspace{-0.5cm}
\includegraphics[width=9cm]{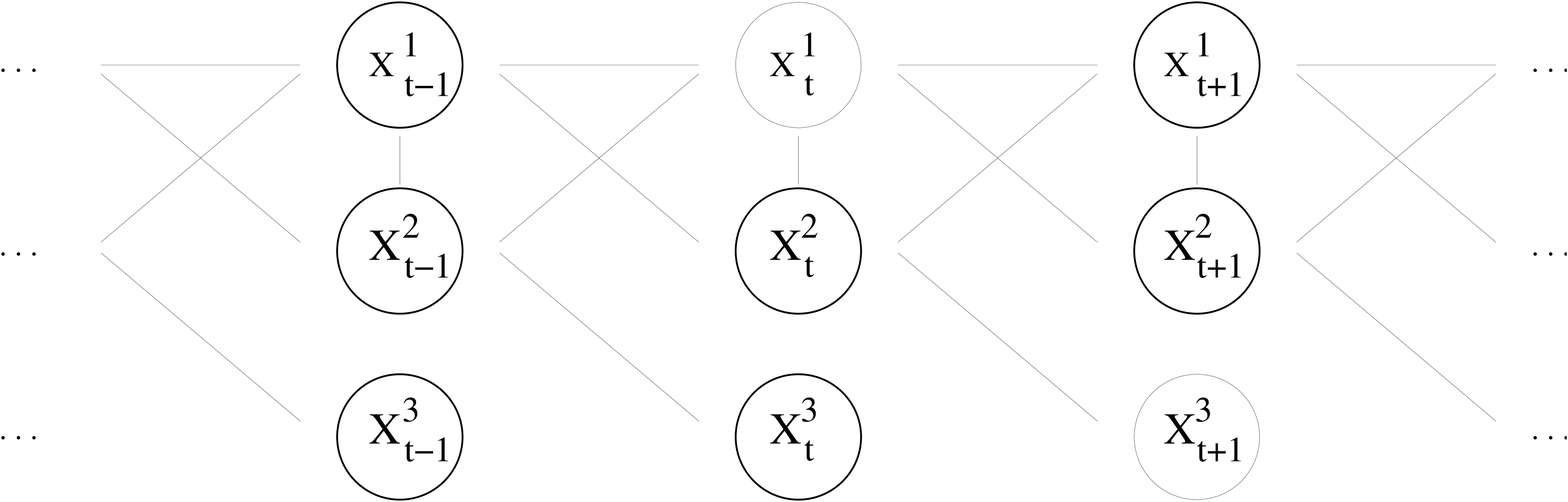}\\
\vspace{1cm}

\includegraphics[width=9cm]{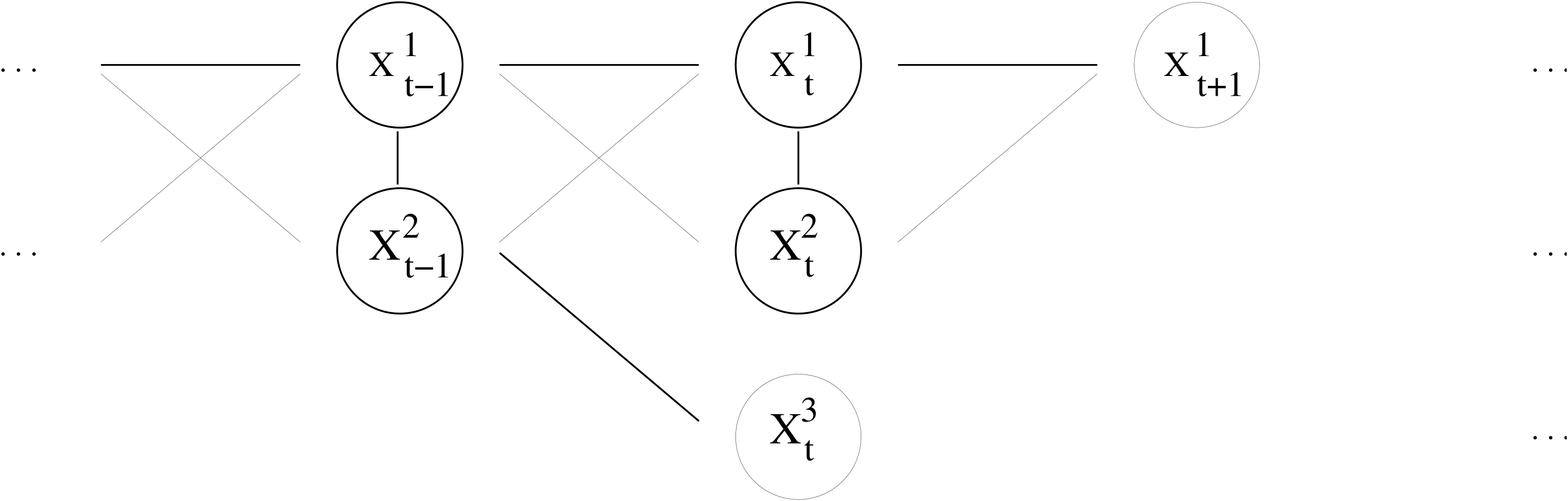}
\end{center} 
\vspace{-0.5cm}
\caption{(A) Moral graph of the DAG in Figure \ref{fig:motifDBN}C. For all $t>1$, the 
parents of  the variable $X^1_t$ are  `married',  that is connected  by an
undirected edge. (B)~Moral graph  of the smallest  ancestral set containing  the variables
  $X^1_{t+1}$,   its  parents  in   the  DAG   in  Figure~\ref{fig:motifDBN}C  and
  $X^3_t$.  As the set $(X^1_t,X^2_t)$ blocks all paths between $X^3_t$ and
$X^1_{t+1}$, thus $\{ X^1_t, X^2_t\}$ separates $X^1_{t+1}$from $X^3_t$ and we have $X^1_{t+1}
\indep X^3_{t}\ \vert \ (X^1_t,X^2_t) $.} 
\label{fig:moralG}
\end{figure}

\vspace{-0.2cm}

\begin{Def}\label{def:gm} \textbf{(Moral graph, Lauritzen \cite{Lauritzen96})} 
The moral graph $\G^m$ of DAG $\G$ 
  is obtained  from $\G$ by first  `marrying' the parents  (draw an undirected
  edge  between each  pair of  parents of each variable $X^i_t$)  and then
  deleting the directions of the original  edges of $\G$. For an illustration, Figure \ref{fig:moralG}A displays the moral graph of the DAG in Figure \ref{fig:motifDBN}C. 
\end{Def}

\begin{Def}\label{def:ancestral} \textbf{(Ancestral set, Lauritzen \cite{Lauritzen96})}
The subset $S$ is ancestral if and  only if, for all $\alpha
\in S$, the parents of $\alpha$ satisfy $ pa(\alpha, \G) \subseteq S$. Hence, for any subset $S$ of vertices,
there  is  a  smallest  ancestral  set  containing $S$  which  is  denoted  by
$An(S)$. Then $\G_{An(S)}$  refers to the graph of  the smallest ancestral set
$An(S)$. See Figure \ref{fig:moralG}B for an illustration.
\end{Def}

\begin{Prop}\label{prop:gAnc} \textbf{(Directed global Markov property, Lauritzen \cite{Lauritzen96}, \newline Corollary 3.23)} Let
  $\proba$ admit a BN representation according to $\G$. Then,
$$E \indep F \ \vert \ S,$$
whenever all paths from $E$ to $F$ intersect $S$ in $(\G_{An(E \cup
F \cup S)})^m$, the moral graph of the smallest ancestral set containing $E \cup
F \cup S$. We say that $S$ \textbf{separates} $E$ from $F$.
\end{Prop}

\paragraph{Sufficient conditions for  DBNs representation}\label{subsec:suffConditions}
We recall here sufficient conditions  under  which the probability
distribution  $\proba$ of process $X$  admits a  BN  representation  according  to a  dynamic
network (e.g. in Figure \ref{fig:motifDBN}C). We   first   assume   that   the   observed   process   $X_t$   is   first-order
Markovian (Assumption~\ref{ass:m1}). That is, the expression level of a gene at a given time $t$ only
depends on the past through the gene expression levels observed at the previous time $t-1$. 
Then we  assume that the  variables observed simultaneously  are conditionally
independent given the past of the process (Assumption
\ref{ass:cindep}). In other words, we consider that time measurements are close enough
so  that   gene  expression level  $X^i_t$  measured at  time  $t$ is  better
explained by the previous time expression levels $X_{t-1}$ than by some current expression
level $X^j_t$. 

\begin{Ass}\label{ass:m1}
The stochastic process $X_t$ is first-order Markovian, that is, 

$$\forall t \geq 3, \ \ X_t \indep X_{1:t-2} \ \ \vert \ \ X_{t-1}.$$ 

\end{Ass}

\begin{Ass}\label{ass:cindep} 
For all $t \geq 1$, the random variables $\{X^i_t\}_{ i \in \p}$ are conditionally
independent given the past of the process $ X_{1:t-1}$, that is,

$$ \forall t \geq 1, \forall i \not= j,\ \ X^i_t \indep X^j_t \ \ \vert \ \ X_{1:t-1}.$$ 

\end{Ass}

Assumptions \ref{ass:m1} and \ref{ass:cindep} allow the existence of a DBN representation
of the distribution 
$\proba$ according to DAG $ \Gfull= ( X , \{(X^j_{t-1}, X^i_t)\}_{i,j \in \p, t>1} )$ 

which contains all the edges pointing out from a variable observed at some time $t-1$ towards a variable observed at the next time $t$ (See Lemma \ref{lem:Gfull} in Appendix \ref{sapp:Lemmas}). 
The direction of the edges according to time guarantees the acyclicity of $\Gfull$.

\subsection{Minimal DAG $\Gmin$}\label{ssec:gmin}
\paragraph{Existence and definition}
Among the  DAGs included in $\Gfull$,   we show that the probability distribution $\proba$ factorizes according to a minimal DAG,  which we denote by $\Gmin$ (See Lemma \ref{lem:interGraph}, Appendix \ref{sapp:Lemmas}).
The set of edges of~$\Gmin$ is exactly the set
 of  full order  conditional dependencies between successive variables given the  past of  the  process as
 set up  in  the   Proposition \ref{prop:Gmin} (See Proof in Appendix~\ref{sapp:proofProps}). 

\newpage
 
\begin{Prop}\label{prop:Gmin}\textbf{(Existence of minimal DAG $\Gmin$,
    the smallest  subgraph \newline of $\Gfull$ allowing DBN modelling)} 
    Let  $\pj \!= \! \p \bs \{j\}$ and $X^{\pj}_{t}\!=\!\{X^k_{t}; k \in\pj\}$ refer to the set $\pj$ of $p-1$
 variables observed at time $t$. Whenever  Assumptions \ref{ass:m1}
  and \ref{ass:cindep} are satisfied, the probability distribution $\proba$ admits a DBN
representation according to DAG $\Gmin$ whose edges describe exactly the
full order conditional dependencies between successive variables $X^j_{t-1}$
  and $X^i_t$ given the remaining  variables $X^{\pj}_{t-1} $ observed at time
  $t-1$,
  \vspace{-0.1cm}
$$ \Gmin = \left (X,\left \{ (X^j_{t-1}, X^i_t); \ X^i_t \not\indep X^j_{t-1} \vert
X^{\pj}_{t-1}\right \}_{i,j \in \p,t \in \n} \right )$$

\noindent Moreover, DAG $\Gmin$ is the smallest subgraph of $\Gfull$ according to which $\proba$ admits a DBN representation.
\end{Prop}

Thus in 
DAG $\Gmin$, the set of parents $pa(X^i_t,\Gmin)$ of a
 variable $X^i_t$ is the smallest subset of $X_{t-1}$ such that the conditional
 densities satisfy $f(X^i_t \vert  pa(X^i_t,\Gmin) )= f(X^i_t \vert X_{t-1}) $. 
The set of parents of a variable can be seen as the only variables on which
this variable depends directly. 
So $\Gmin$ is the  DAG we want to  infer in order to recover
potential  regulation relationships  from gene  expression time  series.  From
Proposition  \ref{prop:Gmin}, any  pair of successive variables $(X^j_{t-1},X^i_t)$
which  are non  adjacent in  $\Gmin$ are  conditionally independent  given the
parents of $X^i_t$. In short, for all $i,j$ in $\p$, for all $t>1$, we have, 
 
 $$(X^j_{t-1},  X^i_t)  \notin E( \Gmin)  \ \ \ \   \Leftrightarrow \ \ \ \  X^i_t \indep X^j_{t-1}\ \vert \  pa(X^i_t,\Gmin) .$$
We
will make use of this result in  Section \ref{sec:Gq} in order to define low
order conditional dependence DAGs for the inference of $\Gmin$.

\paragraph{Minimal DAG $\Gmin$ for an AR(1) process}\label{ssec:ar1}
Consider the following first order auto-regressive model (AR(1)) with a diagonal error covariance
matrix $\Sigma$,

\vspace{-0.4cm}
\begin{eqnarray} 
&&X_1 \sim \mathcal{N}(\mu_1,\Sigma_1) \label{eq:AR1a0} \\
&& \forall t > 1, \ \ \ \ X_t =A X_{t-1} +B +\eps_t,\ \ \eps_t \sim \mathcal{N}(0, \Sigma),\label{eq:AR1a}\\
&& \forall s,t \in \n, \ Cov(\eps_t, \eps_s)= \delta_{ts} \Sigma,\label{eq:AR1b}\\
&& \forall s >t, \ \ \ \ \ Cov(X_t, \eps_s)=0. \label{eq:AR1c}
  \end{eqnarray}

\noindent where $A\!\!=\!\!(a_{ij})_{1\leq i \leq p,1\leq j \leq p}$ is a real matrix of size $ p \! \times \!p$,
$B\!\!=\!\!(b_i)_{1\leq   i   \leq   p}$   is   a   real column   vector,
$\Sigma \!=\!Diag \left (\sigma_{ii}^2 \right )_{1\leq i \leq p}$ is the diagonal error covariance
matrix of size  $p\times p$  and for all $s,\!t$ in $\n$, $\delta_{ts}\!\!=\!\! 1\mskip-5mu\relax\mathrm{l}_{\{s=t\}}$. \!\!\!\! 
\!\!\! Equation~(\ref{eq:AR1c}) implies that the coefficient matrices are
uniquely determined from the covariance function of~$X_t$.

This modelling assumes homogeneity across time (constant matrix $A$) and linearity of
the dependency relationships. 
From (\ref{eq:AR1a})  and (\ref{eq:AR1c}), the model is  first order Markovian
(Assumption  \ref{ass:m1}).
From (\ref{eq:AR1b}), Assumption \ref{ass:cindep} is
satisfied whenever the error covariance matrix $\Sigma$ is diagonal.
Thus from Proposition  \ref{prop:Gmin}, the probability distribution  of the AR(1)
process defined by equations (\ref{eq:AR1a0}-\ref{eq:AR1c}) factorizes  according to the  minimal DAG $\Gmin_{AR(1)}$  whose edges
correspond to the non-zero coefficients of matrix $A$.  Indeed, if matrix $\Sigma$ is diagonal,
each element $a_{ij}$  is the regression coefficient of  the variable $X^i_t$
on $X^j_{t-1}$ given $X^{\pj}_{t-1}$, that is 
$$a_{ij}=
Cov(X^i_t,X^j_{t-1} \ \vert \ X^{\pj}_{t-1} ) / Var(X^j_{t-1} \ \vert \ 
  X^{\pj}_{t-1}).$$

\noindent As process X is Gaussian, the set of null coefficients of matrix $A$ exactly describes the conditional
independencies between successive variables, thus \ \ if $\Sigma$ is diagonal, we have,
$$a_{ij}=0\ \ \ \Leftrightarrow \ \ \ \forall t >1, \ \ X^i_t \indep X^j_{t-1} \vert
X^{\pj}_{t-1}. $$

\noindent Finally, DAG $\Gmin_{AR(1)} $ has an edge between two successive
variables  $X^j_{t-1}$ and  $X^i_t$,  for all  $t>1$,  whenever the  coefficient
$a_{ij}$ of the matrix $A$ differs from zero,
\begin{equation}\label{eq:GminAR}
\Gmin_{AR(1)}:=\left  (X, \left  \{(X^j_{t-1}, X^i_t)  \text{ such  that  } a_{ij}\not=
0; \ t>1,\ i,j \in \p \right \} \right ).
\end{equation}

\noindent As an illustration, any AR(1) process whose matrix $\Sigma$ is diagonal and
 matrix~$A$ has the following form,
$$ A = \left (
\begin{array}{ccc}
a_{11} & a_{12}& 0\\
a_{21}& 0&0\\
0 & a_{32}& 0
\end{array}
\right ),$$
admits a
BN representation according to the dynamic network of
Fig.\ref{fig:motifDBN}C~($p\!=~$\!\!\!$~3$).

\section{Introducing $q^{th}$ order dependence DAGs $\gq$ for~DBNs}\label{sec:Gq}

In this paper, we propose to use the DBN modelling according to DAG $\Gmin$ (introduced in  Proposition \ref{prop:Gmin}) to model genetic regulatory networks from gene expression time series.
Reverse discovery of  DAG $\Gmin$ requires to determine, for each variable $X^i_t$, the set of variables
$X^j_{t-1}$ observed at time $t-1$ on which variable $X^i_t$ is conditionally dependent  given the remaining variables~$X^{\pj}_{t-1}$. 
However, even under the time homogeneity assumption discussed in the introduction, 
standard estimation methods do not allow us 
 to infer the parameters of a regression model  for $p$ genes (\ie$p^2$ possible edges) from $np$ measurements.
 We still have
to face the `curse of dimensionality' since the number of genes $p$, is much higher than the number
of measurements~$n$. 

In order to reduce the dimension, we approximate DAG $\Gmin$ by   $q^{th}$ order conditional dependence DAGs  $\gq$  ($q < p$). To such an end, we extend to DBNs the approach based on the consideration of low order independencies introduced by Wille et al.\cite{Wille04,Wille06} for GGM approximation (See more details on low order independence graphs for GGMs in Section \ref{sec:intro}). After defining $q^{th}$ order conditional dependence DAGs $\gq$   for DBNs, we investigate the manner  in which they allow us to approximate  the DAG $\Gmin$ describing full order conditional dependencies.

\subsection{ DAG $\gq$ definition}

Let $q$ be smaller than $p$. In the  $q^{th}$ order dependence DAG $\gq$, 
whenever there exists a subset $X^Q_{t-1} $ of $q$ variables among the set of $p-1$ variables
$X^{\pj}_{t-1}$   such  that   $X^j_{t-1}$  and   $X^i_t$   are  conditionally
independent given $X^Q_{t-1} $, no edge is drawn between the two successive variables $X^j_{t-1}$ and $X^i_t$.
In short, DAGs $\gq$ are defined as follows,

\begin{Def}\label{def:gq} \textbf{$q^{th}$-order conditional dependence DAG $\gq$}
$$ \forall q \!<\! p, \ \gq \!\!=\!\! \left ( \!\! X,\left \{ \! (X^j_{t-1}, X^i_t); \ \forall Q \subseteq \pj,
\vert Q \vert =q, X^i_t \not\indep X^j_{t-1} \vert
X^Q_{t-1}\right \}_{i,j \in \p,t \in \n} \!\!  \right )\!\!.$$
\end{Def}

 DAGs $\gq$ offer a way of
producing dependence relationships between the
variables,  but they are no longer associated with a BN representation
which would call  for more global relationships.
Note that the definition of $q^{th}$ order partial dependence DAG $\gq$ is based on exact  $q^{th}$ order independencies (not on all partial independencies lower than $q$ as in  the partial order correlation network used by  Wille and
    B\"uhlmann \cite{Wille06}). 
  Indeed, we consider that including only  the $q^{th}$ order dependencies  better reflects the true DAG $\Gmin$. In particular, for $p$ variables, DAG $ \mathcal{G}^{(p-1)}$ is DAG $\Gmin$. This definition is possible for DBNs  because  dynamic modelling essentially differs from static correlation network modelling\footnote{In particular, contrary to the case of correlation network,  the `` V '' structures (or structures with multiple parents)  do not generate spurious edges in the case of DBN since the definition of the DAG $\Gmin$ defining full order dependencies does not allow edges between variables observed at the same time. Thus, for instance,  when considering  the following  ``~V~'' structure $X^j_{t-1} \rightarrow X^i_t \leftarrow X^k_{t-1}$,  no spurious edge can be inferred between the variables  $X^j_{t-1}$ and $X^k_{t-1}$.}.

In general,
DAGs $\gq$ differ from DAG $\Gmin$. For instance, the approximation
of the DAG of Figure \ref{fig:motifDBN}C by the $1^{st}$ order
conditional dependence DAG  may give rise to the spurious edge $X^3_{t} \rightarrow
X^1_{t+1}$, for all $t<n$ (See Figure \ref{fig:Gq}). 
Indeed, $X^1_{t}$ (resp. $X^2_{t}$) does not  separate $X^1_{t+1}$ from $X^3_{t}$ in the smallest moral graph containing the variables $X^1_{t+1} \cup X^3_{t} \cup X^1_{t}$  (resp. $X^1_{t+1} \cup X^3_{t} \cup X^2_{t}$)
 displayed in Figure \ref{fig:moralG}B.
Nevertheless, if the  vertices of
$\Gmin$ have few parents, 
DAGs $\gq$ bring relevant information about the  topology of $\Gmin$,  even for  small values of  $q$. 
In the following, we give characterizations of low order conditional
dependence DAGs $\gq$  and analyze the accuracy of the approximations they offer.

\begin{figure}
\begin{center}
 \includegraphics[width=9cm]{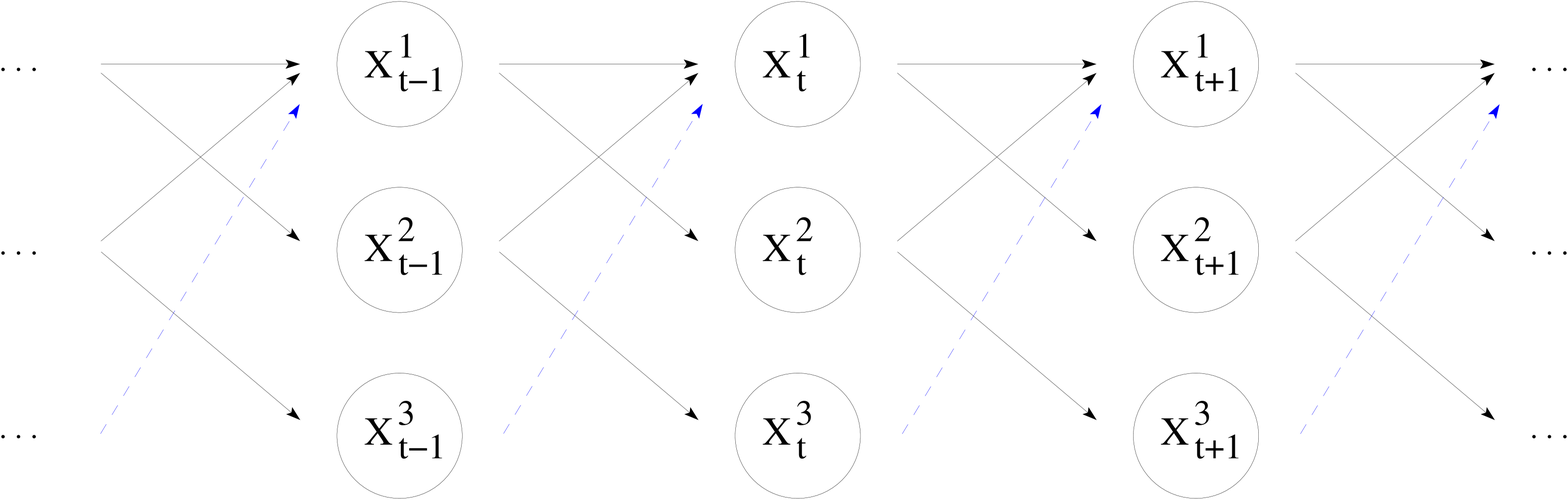} 
\end{center}
\vspace{-0.5cm}
\caption{First-order conditional dependence DAG $\gun$ (obtained from the DAG in
Figure~\ref{fig:motifDBN}C). The spurious dashed arrow may appear in $\gun$. }
\label{fig:Gq}
\end{figure}

\vspace{-0.2cm}

\subsection{A restricted number of parents}

In some known gene regulation mechanisms, it is the case that a few genes regulate many other
genes (e.g. the single input modules in the transcriptional regulatory network of
\textit{S. Cerevisiae} \cite{Lee02}). However, we do not expect a single gene to be regulated by many genes at the same time. So the number of parents in gene
interaction networks is expected to  be relatively small.  In this section, we
analyze the  properties of  $\gq$ when  the number of  parents in  $\Gmin$ is
lower than $q$. 

Let us denote by $\NpaxitGmin$ the number of parents of $X^i_t$ in
DAG  $\Gmin$  and  $ \NpaMaxGmin$  the  maximal  number  of parents  of  any
variable $X^i_t$ in $\Gmin$,
\vspace{-0.6cm}
 
$$ \NpaxitGmin=\left \vert pa(X^i_t, \Gmin) \right \vert, \hspace{0.6cm}
 \NpaMaxGmin =  \mathop{\text{Max}}_{i \in \p, t \in \n} \left (\NpaxitGmin \right ). $$

\vspace{-0.2cm}

\noindent The next  results hold when the number  of parents in $\Gmin$
is restricted. 

\vspace{-0.1cm}

\begin{Prop}\label{prop:infq}
If $\NpaxitGmin \leq q$ then we have,
\vspace{-0.2cm}

$$\left \{ (X^j_{t-1},X^i_t)
\notin E(\Gmin)\right \}\ \Rightarrow \ \left \{ (X^j_{t-1},X^i_t) \notin
E(\G^q)\right \}.$$
\vspace{-0.2cm}

\end{Prop}

\vspace{-1cm}

\begin{Cor}\label{cor:GqInGmin}
 For all $ q \geq \NpaMaxGmin$, we have $ \Gmin \supseteq \gq.$
\end{Cor}

\vspace{-0.6cm}

\begin{Prop}\label{prop:G1}  Let  $X$  be  a  Gaussian
  process.\! If $\NpaMaxGmin \leq 1$ then $\Gmin\!= \!\gun$.  
\end{Prop}

Consider a variable $X^i_t$ having at most $q$ parents in $\Gmin$ ($q <
p$). Let $X^j_{t-1}$ be a variable observed at the previous time $t-1$ and having
no edge pointing towards $X^i_t$ in $\Gmin$. In the moral  graph of the smallest ancestral set containing $X^i_t
\cup X^j_{t-1} \cup \  \text{pa}(X^i_t, \Gmin)$, the set of  parents $\text{pa}(X^i_t, \Gmin)$
separates $X^i_t$ from $X^j_{t-1}$.  From Proposition
\ref{prop:gAnc}, we have $X^i_t\indep~X^j_{t-1} \ \vert \ \text{pa}(X^i_t, \Gmin) $. 
The number of parents $\text{pa}(X^i_t,
\Gmin)$ is smaller than $q$, so the edge $X^j_{t-1} \rightarrow X^i_t$ is not in
$\Gq$. This establishes Proposition \ref{prop:infq}. 
Consequently,  if the  maximal  number of  parents  in $\Gmin$  is lower  than
$q$, then $\gq$ is included in $\Gmin$ (Corollary \ref{cor:GqInGmin}).  In this case, $\gq$ does not contain spurious edges.

The converse  inclusion relationship  is not  true in  general\footnote{As an illustration, let
$X^j_{t-1}  \!\rightarrow\!  X^i_t$  be  an  edge of  $\Gmin$  then in essence (See Prop \ref{prop:Gmin}) $X^i_t$  and
$X^j_{t-1}$  are   conditionally  dependent  given   the  remaining  variables
$X^{\pj}_{t-1}$.  There may  however exist a subset of  $q$~variables $X^Q_{t-1}$,
where $Q$  is a subset of  $\p \bs \{j\}$ of  size $q$, such  that $X^i_t$ and
$X^j_{t-1}$  are  conditionally  independent   with  respect  to  this  subset
$X^Q_{t-1}$. Indeed, even  though the topology of $\Gmin$  allows us to establish
some  conditional independencies,  DAG  $\Gmin$ does  not  necessarily allow  to
derive  all of  them.   Two variables can be
conditionally independent given a subset of variables whereas this subset does
not separate these  two variables in~$\Gmin$.}.  Nevertheless,  if each variable
has at most \textit{one} parent, the converse inclusion $\Gmin \subseteq \gun$
is true if  the process is Gaussian and  $q=1$ (Proposition \ref{prop:G1}, see
proof in Appendix \ref{sapp:proofProps}).  
 At  a higher  order,  we need  to  assume that  all
conditional independencies  can be derived  from $\Gmin$, that is  $\proba$ is
\textit{faithful} to $\Gmin$.

\subsection{Faithfulness} 
\begin{Def}\label{Def:faithful}\textbf{(faithfulness,                   Spirtes
    \cite{Spirtes93})} A  distribution $\proba$ is \textbf{faithful}  to a DAG
  $\G$  if  all and only the independence  relationships true in $\proba$
  are  entailed   by  $\G$  (as   set  up  in   Proposition~\ref{prop:gAnc}). 
\end{Def} 

\begin{Thm}\label{thm:MesZero}\textbf{(Measure zero for unfaithful Gaussian (Spirtes \cite{Spirtes93}) and discrete (Meek
    \cite{Meek95})   distributions)}   Let   $\pi^{\mathcal{N}}_{\G}$   (resp. 
  $\pi^{\mathcal{D}}_{\G}$)  be  the set  of  linearly independent  parameters
  needed to  parameterize a  multivariate normal distribution  (resp.  discrete
  distribution)  $\proba$ which  admits  a factorization  according  to a  DAG
  $\G$. 
  The set of  distributions which are unfaithful to $\G$  has measure zero with
  respect  to  Lebesgue  measure  over $\pi^{\mathcal{N}}_{\G}$  (resp.   over
  $\pi^{\mathcal{D}}_{\G}$).
\end{Thm} 

From Definition \ref{Def:faithful}, whenever the distribution  $\proba$ is faithful  to $\Gmin$, any  subset $X^Q_{t-1}
\subseteq  X_{t-1}$,  with  respect  to  which  $X^i_t$  and  $X^j_{t-1}$  are
conditionally  independent, separates  $X^i_t$  and $X^j_{t-1}$  in the  moral
graph  of the  smallest ancestral  set containing  $X^i_t \cup  X^j_{t-1} \cup
X^Q_{t-1}$.  Under  this assumption, we  can derive interesting  properties on
$\Gmin$ from the topology of low  order dependence DAGs $\gq$.  As there is no
way  to assess  a  probability distribution  to  be faithful  to  a DAG,  this
assumption    has    often    been    criticized.     However,    Theorem
\ref{thm:MesZero},  established  by   Spirtes  \cite{Spirtes93}  for  the Gaussian
distribution  and extended  to  discrete distributions  by Meek  \cite{Meek95},
makes this assumption reasonable at least in a measure-theoretic sense.  
Moreover this assumption remains  very reasonable in a modelling framework where the network to be  inferred describes actual interaction relationships.
The next propositions are derived from the faithfulness of the distribution $\proba$ to $\Gmin$ (See
proofs 
 in Appendix \ref{sapp:proofProps}).

\begin{Prop}\label{prop:GminInGq} Assume  $\proba$  is faithful  to
  $\Gmin$.  For all $q <p$, we have $\Gmin \subseteq \gq $.  
\end{Prop}

\begin{Cor}\label{cor:NMaxGmin} Assume  $\proba$  is faithful  to
  $\Gmin$. \!\!\! For all $q \!\geq \! \NpaMaxGmin$, we have $\Gmin =\gq$\!.  

\end{Cor}

\begin{Prop}\label{prop:NpaGq}  Assume   $\proba$  is  faithful   to  $\Gmin$. \\
  \indent  If  $ N_{pa}(X^i_t,  \gq) \leq q  $ then $\  (X^j_{t-1},X^i_t) \in
  E(\gq) \ \Rightarrow \ (X^j_{t-1},X^i_t) \in E(\Gmin)$.  

\end{Prop}

\begin{Cor}\label{cor:NMaxGq} Assume $\proba$ is faithful to $\Gmin$.  For all
  $ q \geq N_{pa}^{Max}(\gq) $, 
  $\Gmin = \gq $.  

\end{Cor} 
 
 Whenever $\proba$  is faithful  to
  $\Gmin$, DAG $\gq$ contains DAG $\Gmin$ (Proposition \ref{prop:GminInGq}).  
Even though we expect the number  of parents in a gene interaction networks to be  bounded aboce,
 the exact  maximal number of parents  $\NpaMaxGmin$ remains mostly
unknown.   However, we show that the  edges of DAG  $\gq$ pointing towards a  variable having
less than $q$ parents in $\gq$ are edges of $\Gmin$ too (Proposition \ref{prop:NpaGq}). 
Thus, if  $\proba$ is  faithful to  $\Gmin$, knowledge of the topology of DAG $\gq$ only allows us to ascertain some  edges of DAG $\Gmin$.
From Propositions \ref{prop:GminInGq} and 
  \ref{prop:NpaGq}, we establish that both DAG $\gq$ and DAG  $\Gmin$ exactly coincide if any node of $\gq$ has less than $q$ parents (Corollary \ref{cor:NMaxGq}).

\section{\textit{G1DBN}, a procedure for DBN inference}\label{sec:estim} 

We introduced and characterized the  $q^{th}$ order dependence DAGs $\gq$, for
all $  q< p$, for dynamic  modelling.  We now  exploit our results to  develop a
non-Bayesian inference method  for DAG $\Gmin$ defining a DBN representation for process $X$.  Let $q_{max}$  be the maximal
number  of parents  in  $\Gmin$.  From  Corollary \ref{cor:NMaxGq},  inferring
$\Gmin$  amounts to  inferring  $\G^{(q_{max})}$.  
However,  the inference  of
$\G^{(q_{max})}$ requires to  check, for each pair $(i,j)$,  if there exists a
subset  $Q \subseteq  \pj$  of  dimension $q_{max}$  such  that $X^i_t  \indep
X^j_{t-1} \vert X^Q_{t-1}$ for all $t>1$. So, for each pair $(i,j)$, there are
$\left  (^{q_{max}}_{p-1}  \right  )$  potential  sets  that  can  lead  to
conditional independence. 
To  test  each conditional  independence given  any
possible  subset of  $q_{max}$  variables  is questionable  both  in terms  of
complexity and multiple testings.

To circumvent these issues, we propose to
exploit  the  fact  that the  true  DAG  $\Gmin$  is  a subgraph  of  $\gun$
(Proposition \ref{prop:GminInGq}) in order to develop an inference procedure for $\Gmin$. 
Indeed, the inference of $\G^{(1)}$
is  both  faster (complexity)  and  more accurate  (number  of tests). 
Thus we introduce a 2 step-procedure 
for DBN inference. In the first step, we infer  the $1^{st}$ order dependence DAG $\gun$, then we infer DAG  $\Gmin$  from the  estimated  DAG
$\hat{\G}^{(1)}$. 
This 2 step-procedure, summarized in Figure \ref{fig:algo}, 
is implemented in a R package `G1DBN' \cite{G1DBN}
freely available from the Comprehensive R Archive Network. 

\subsection {Step 1: inferring $\gun$}\label{ssec:g1hat}

We evaluate the \textit{likelihood} of an edge $(X^j_{t-1},X^i_t)$ 
by measuring the conditional dependence
between the variables $X^j_{t-1}$ and $X^i_t$ given any variable $X^k_{t-1}$. 
Assuming linear dependencies, we consider   the  partial regression coefficient $a_{ij  \vert k}$ defined as follows,

\begin{equation*}\label{eq:aijk} X^i_t = m_{ijk}+a_{ij\vert k} X^j_{t-1} +
  a_{ik\vert j} X^k_{t-1} +\eta^{i,j,k}_t, 
 \end{equation*} 

\noindent where  the rank  of the matrix  $(X^j_{t-1}, X^k_{t-1})_{t  \geq 2}$
  equals $2$ and the errors $\{\eta^{i,j,k}_t\}_{t\geq 2}$ are centered, have same variance
  and are not correlated. 

We measure the conditional dependence
between the  variables $X^j_{t-1}$ and $X^i_t$ given  any variable $X^k_{t-1}$
by testing the null assumption $\mathcal{H}^{i,j,k}_0$: ``$a_{ij\vert k}= 0$''.
To such  an aim, we  use one out  of three M-estimators for  this coefficient:
either the familiar Least Square (LS) estimator, the \textit{Huber} estimator,
or  the \textit{Tukey  bisquare}  (or \textit{biweight})  estimator.  The  two
latter  are  robust estimators  \cite{Fox02}.   Then  for  each $k\not=j$,  we
compute the estimates  $\hat{a}_{ij \vert k}$ according to  one of these three
estimators and  derive the p-value  $p_{ij,k}$ from the  standard significance
test:

\begin{equation}\label{eq:testS1}
under \ (\mathcal{H}^{i,j,k}_0): \ ``\ a_{ij\vert k}=0\ \text{''}, \ \ \ \ \ \ \ \ \frac{\hat{a}_{ij\vert k}}{\hat{\sigma}(\hat{a}_{ij\vert k})} \sim t(n-4), 
\end{equation}
\noindent where $t(n-4)$ refers to a student probability distribution with $n-4$
degrees of freedom
and $\hat{\sigma}(\hat{a}_{ij\vert k})$ is the variance estimates for  $\hat{a}_{ij \vert k}$.

Thus, we assign a score  $S_1(i,j)$ to each potential edge $(X^j_{t-1},X^i_t)$
equal to the  maximum $\mathop{Max}_{k \not=j} (p_{ij\vert k}  )$ of the $p-1$
computed  p-values,  that is  the  most  favorable  result to  $1^{st}$  order
conditional  independence.  This procedure  does not  derive p-values  for the
edges but  allows to order the possible  edges of DAG~$\gun$  according to how
likely they are. The smallest scores  point out the most significant edges for
$\gun$. The inferred  DAG $\gunhat$ contains the edges assigned  a score below a
chosen threshold $\alpha_1$.



\begin{figure}
  \centering
  \begin{tabular}{c}\hline
\begin{minipage}{13.5cm}
  \ttfamily
  \small
Choose either LS, \!\!Huber or Tukey estimator and set $\alpha_1$ \!\! and \!\! $\alpha_2$ thresholds.\\
\vspace{0.1cm}
 \textbf{Step 1: inferring $\gun$.} \\
\hspace{0.5cm} For all $i \in \p$,\\
\hspace{0.5cm} \hspace{0.5cm} For all $j \in \p$, for all $k \not= j$, compute
  the p-value $p_{ij\vert k}$ from (\ref{eq:testS1}),\\
\hspace{0.5cm}\hspace{0.5cm} $S_1 (i,j) = \mathop{Max}_{k \not=j} (p_{ij\vert k} )$.\\
\hspace{0.5cm} $E(\hat{\G}^{(1)})=\{ (X^j_{t-1},X^i_t)_{t>1}; \ i,j \in \p, \text{ such that } S_1(i,j) < \alpha_1 \}$.  \\
\vspace{0.1cm}
\textbf{Step 2: inferring $\Gmin$ from  $\hat{\G}^{(1)}$.}\\
\hspace{0.5cm} If $N_{pa}^{Max} (\hat{\G}^{(1)}) \sim n-1$, choose a higher
threshold $\alpha_1$ and go to Step1.\\
\hspace{0.5cm} For all $\!i\!$ such that $\!\! N_{pa}(X^i_t,\hat{\G}^{(1)}) \!\geq \! 1$, $\!\!\!\!$ compute the
  p-value $p_{ij}^{(2)}\!\!$ from (\ref{eq:testS2}). \\
\hspace{0.5cm} \hspace{0.5cm} $S_2 (i,j)= \left \{
\begin{array}{ll}
 p_{ij}^{(2)}&for \ all \ i,j \in \p \ such \ that \ (X^j_{t-1},X^i_t)_{t>1} \in \hat{\G}^{(1)},\\
1&otherwise.\\
\end{array}
\right .$\\
\hspace{0.5cm} $E(\Gmin)= \{ (X^j_{t-1},X^i_t)_{t>1 }; i \in \p, (i,j) \in \p  \text{ such that }  S_2(i,j) < \alpha_2 \}$.
\end{minipage}
\\\hline
  \end{tabular}
  \caption{Outline of the 2 step-procedure \textit{G1DBN} for DBN inference.} 
  \label{fig:algo}
\end{figure}

\vspace{-0.2cm}

\subsection {Step 2: inferring $\Gmin$ from $\gun$}\label{ssec:inferGmin}
We use the inferred DAG $\hat{\G}^{(1)}$ as a reduction of the search
space.  Indeed,   from  faithfulness, we know that  $\Gmin   \subseteq  \gun$  (Proposition
\ref{prop:GminInGq}). Moreover, when DAG $\Gmin$ is sparse, there are far fewer edges in $\gun$ than in the
complete DAG~$\Gfull$ defined in Section \ref{subsec:suffConditions}.  Consequently, the number of parents of each
variable in $\gunhat$ is much smaller than $n$. Then model selection can be
carried  out  using  standard  estimation  and  tests  among the  edges  of
$\hat{\G}^{(1)}$. For each pair $(i,j)$ such that the set of edges $(X^j_{t-1},X^i_t)_{t>1} $ is in
$\hat{\G}^{(1)}$, we denote by $a_{ij}^{(2)} $ the regression
coefficient,

\begin{equation}\label{eq:aijpa} 
X^i_t = \ m_{i}+ \sum_{j \in \text{pa}(X^i_t,\hat{\G}^{(1)})} a_{ij}^{(2)}  \ X^j_{t-1} +\eta^{i}_t, 
 \end{equation}
\noindent where the rank of the matrix $(X^j_{t-1})_{t \geq 2, j  \in \text{pa}(X^i_t,\hat{\G}^{(1)})}$
is $\vert pa(X^i_t,\hat{\G}^{(1)}) \vert$ and the errors $\{\eta^{i}_t\}_{t \geq
  2}$ are centered, have the same variance, and are not correlated.  
We assign to  each edge of $\hat{\G}^{(1)}$ a score  $S_2(i,j)$ equal to the
p-value $ p_{ij}^{(2)} $ derived from the significance test,
\vspace{-0.8cm}

\begin{equation}\label{eq:testS2}
\text{under}  \ (\mathcal{H}^{i,j}_0):  \  `` \ a_{ij}^{(2)}=0\ \text{''},   \  \ \  \ \  \ \  \
\frac{\hat{a}_{ij}^{(2)}  }{\hat{\sigma}(\hat{a}_{ij}^{(2)})} \sim t(n-1-\vert
pa(X^i_t,\hat{\G}^{(1)} )\vert ). 
\end{equation}
\vspace{-0.4cm}

\noindent The  score $S_2(i,j)=1$  is assigned  to the edges  that are  not in
$\hat{\G}^{(1)}$.
The smallest scores indicate the most significant edges. The
inferred DAG for $\Gmin$ contains those edges whose score is below a
chosen threshold $\alpha_2$.

When  $\Gmin$  is  sparse, Step 1 of \textit{G1DBN}  inference procedure gives already  a  good
estimation of $\Gmin$ (See  Precision-Recall curves obtained for simulated data in Figure \ref{fig:PRsimus}). 
Even better results can be obtained  with the 2 step-procedure which requires
to tune two parameters $\alpha_1$ and $\alpha_2$.  Parameter $\alpha_1$ is the
selection threshold of the edges of $\gunhat$ in Step 1 (that is the dimension
reduction threshold), whereas parameter  $\alpha_2$ is the selection threshold
for the  edges of $\Gmin$ among  the edges of DAG  $\hat{\G}^{(1)}$. 

\vspace{-0.4cm}

\subsection {Choice of the thresholds}\label{ssec:choiceThres}

The choice of thresholds is  often  something non trivial, especially when using multiple testing.  
However, Step 1 of the procedure is conservative by construction. Indeed, the definition of score $S_1$ (equal to the maximum of $p-1$ $p$-values computed for testing 1st-order conditional independence) clearly supports the acceptation of the null assumption, \ie the absence of an edge.
Standard approaches for  multiple testing correction do not apply to choose  $\alpha_1$ threshold.  Thus we introduce  a heuristic approach to choose $\alpha_1$~threshold  which is detailed in Supplementary Material \cite{SuppMat}, Section B
. Overall, $\alpha_1$ threshold  is chosen so that,  after the Step 1, the number of genes having exactly one parent  in DAG $\gun$ predominates.

The choice of  $\alpha_2$ threshold  is less problematic. Indeed, the second Step of the inference procedure is a standard multivariate regression. Then the usual thresholds  1\%, 5\%  or 10~\% can be chosen or even a lower threshold when a low number of edges is wanted. However, a large number of tests are computed (as many as  edges in DAG~$\gun$). In such multiple testing situations, a set of the predictions are \textit{expected}  to be false and it is useful to control this.
We  control the expected \textit{proportion} of false positives edges, \ie  the False Discovery Rate (FDR) with the approach   introduced by Benjamini and Hochberg\footnote{
Let $m$ be the number of remaining edges after Step $1$, 
then Step 2 requires to compute $m$ tests.  \  \ Choose a maximal FDR level $q$ and order the set of $m$ observed $p$-values: \ \ \ $p_{(1)}  \leq \cdots \leq p_{(i)} \leq \cdots \leq p_{(m)}$.  Then reject  the null assumption (H$_0^{(i)}$: ``Edge $i$ is not DAG~$\Gmin$'')  for all $i \leq k$ where $k$ is defined as follows:
$ k= max \left \{ i: p_{(i)} \leq \frac{i}{m} q\right \}.$
\noindent If no such $i$ exists, reject no hypothesis. Benjamini and Hochberg (1995) showed that this procedure ensures  the FDR is lower than $q \frac{m_0}{m}\leq q$ where $m_0$ is the number of true null hypotheses.} \cite{Benjamini95}.

\vspace{-0.4cm}

\subsection {Complexity of the algorithm}
The complexity of this algorithm is $O(p^3)$. However the scores $\left (S_1(i,j)\right )_{j\in P}$  of the incoming edges of each target gene $i$
can be computed separately by using parallel run.  This option is available in the R package \textit{G1DBN} by  specifying the target gene $i$ in the function  \textit{DBNScoreStep1} dedicated to the Step 1 computation.

All the computations were performed on Redhat WS 4 AMD opteron 270 (2GHz). 
The computation time mostly depends on the number of TF genes, \ie the genes allowed to be parents in the DAG to be  inferred. For an illustration based on DBN inference performed from  
 a real data set by Spellman~\cite{Spellman98} containing $786$ target genes in Section \ref{ssec:appliLevure}, the computation of Step 1  required $7$ minutes when the set of  possible TF genes  was restricted to  $18$ genes (resp. $4$ minutes with the lasso \cite{Tibshirani96} and $7$ seconds with the shrinkage procedure \cite{OpgenStrimmer07}, which are two alternative approaches for DBN inference introduced in Section~\ref{ssec:refMethods}). When  all the 786 genes can be TFs, the computation was parallel run and required $19$ minutes by target gene with G1DBN (resp. $8$~minutes by target gene with the lasso and $5$ minutes for the whole set of 786 target genes with the shrinkage procedure).  Step 2 of \textit{G1DBN} is very quick and requires less than 5 seconds for the 786-TF study.
Despite the need for more time,  inference with \textit{G1DBN} for a data set containing 800 genes is fully computable, especially when parallel running.

\section{Validation}\label{sec:valid}

\subsection{Comparison with two reference methods} \label{ssec:refMethods}

 We compare the  \textit{G1DBN}    inference procedure with two  reference  methods  for  model selection for multivariate AR(1) process: the shrinkage approach by Opgen-Rhein and Strimmer  \cite{OpgenStrimmer07} and the lasso (Least Absolute Shrinkage and Selection Operator) introduced by Tibshirani  \cite{Tibshirani96}.    
Opgen-Rhein  and  Strimmer   recently proposed  a  model selection procedure  based on an  analytic approach using James-Stein-Type shrinkage.   The procedure consists of first computing the  partial correlation coefficients,  $r(X^i_t,X^j_{t-1} \vert X^{P_j}_{t-1})$, from the
shrinkage  estimates of the  partial regression  coefficients, and  second, selecting the edges with a \textit{local} false discovery rate approach \cite{efron05}. 
Shrinkage inference is performed using the R code for shrinkage estimation\footnote{available at \texttt{http://strimmerlab.org/software.html}.}  
by Opgen-Rhein and Strimmer.
  
The  lasso (also called L1 shrinkage) combines shrinkage  and model selection. The lasso estimates are obtained by minimizing the residual sum of squares subject to the sum of the absolute values of the coefficients being less than a constant. This approach offers the  advantage that it automatically sets many regression coefficients  to zero.   We performed the  lasso  with the R  package LARS  developped by Efron et al. \cite{efron04}. 


\subsection{Simulation study}\label{ssec:simu}


As the discovery of genetic regulatory interaction is a field in progress, validation of predictions made on real gene expression data is only partial, which may render the estimation of true and false positive detection rate not fully reliable \cite{Husmeier03}. Thus we first investigate the accuracy of   \textit{G1DBN},  the shrinkage  and the lasso inference procedures on simulated data. 


\paragraph{Data generation} 
We generated 100 random time series according to  a  multivariate AR(1) model  defined by parameters ($A_{[p \times p]},\!B, \Sigma$)
for $p\!=\!50$ genes. Since gene  regulation networks are sparse, each  matrix $A$ contains 5\%  of non zero coefficients.
 While keeping the number of parents low, this does not prevent a vertex from having more than one parent.
Non   zero   regression coefficients   $a_{ij}$,  mean 
 coefficients~$b_i$  and error variances $\sigma_i $ were    drawn from  uniform distributions ($a_{ij}, b_i \!\sim\! \mathcal{U}([-0.95;-0.05]  \!\cup \![0.05; 0.95]), \sigma_i  \!\sim\! \mathcal{U}[0.03, 0.08]$).  Time   series were generated under   the corresponding multivariate AR(1) models for $n\!=\!20$ to~$50$.


\paragraph{Evaluation  based on PR curves}
We evaluated the performance of DBN inference procedures using  the \textit{Precision-Recall (PR)} curve as plotted in Figure~\ref{fig:PRsimus}. PR curves show the precision, equal to the Positive Predictive Value (PPV) on the ordinate against  the recall, equal to the  power,  on  the abscissa. PR  curves are drawn  by first ordering the edges by decreasing significance, 
and then by computing the PPV and power for the first selected edge and for each newly included edge successively. 
We recall the next definitions, 
\begin{eqnarray*}
\text{Positive Predictive Value (PPV) }&=& \ \ \text{ True Discovery Rate (TDR)}\\
& = & \ \  \text{ 1- False Discovery Rate (FDR)} \\
&=& \ \  \frac{TP}{TP+FP}\\
\text{Recall \  =  \ Sensitivity \ = \ Power }  &=&\ \  \frac{TP}{TP+FN}
\end{eqnarray*}
\noindent where TP refers to the number of true positive edges,  \ie  the number of edges \newline which are selected by the inference procedure and actually belongs to the  true  DAG 
(used for simulating the data);
FP  refers to the number of false positive edges, \ie the edges which are  selected by the procedure but are not in the  true DAG and FN refers to the number of false negative edges, \ie the number of edges which are not selected by the procedure but are in the true DAG.


\begin{figure}
\begin{picture}(20,10)
\put(0,-5){\Large A}
\put(0,-258){\Large B}
\end{picture}

    \centering
\vspace{-2.8cm}
\hspace{-0.7cm}\includegraphics[width=12cm,height=12cm,angle=-90]{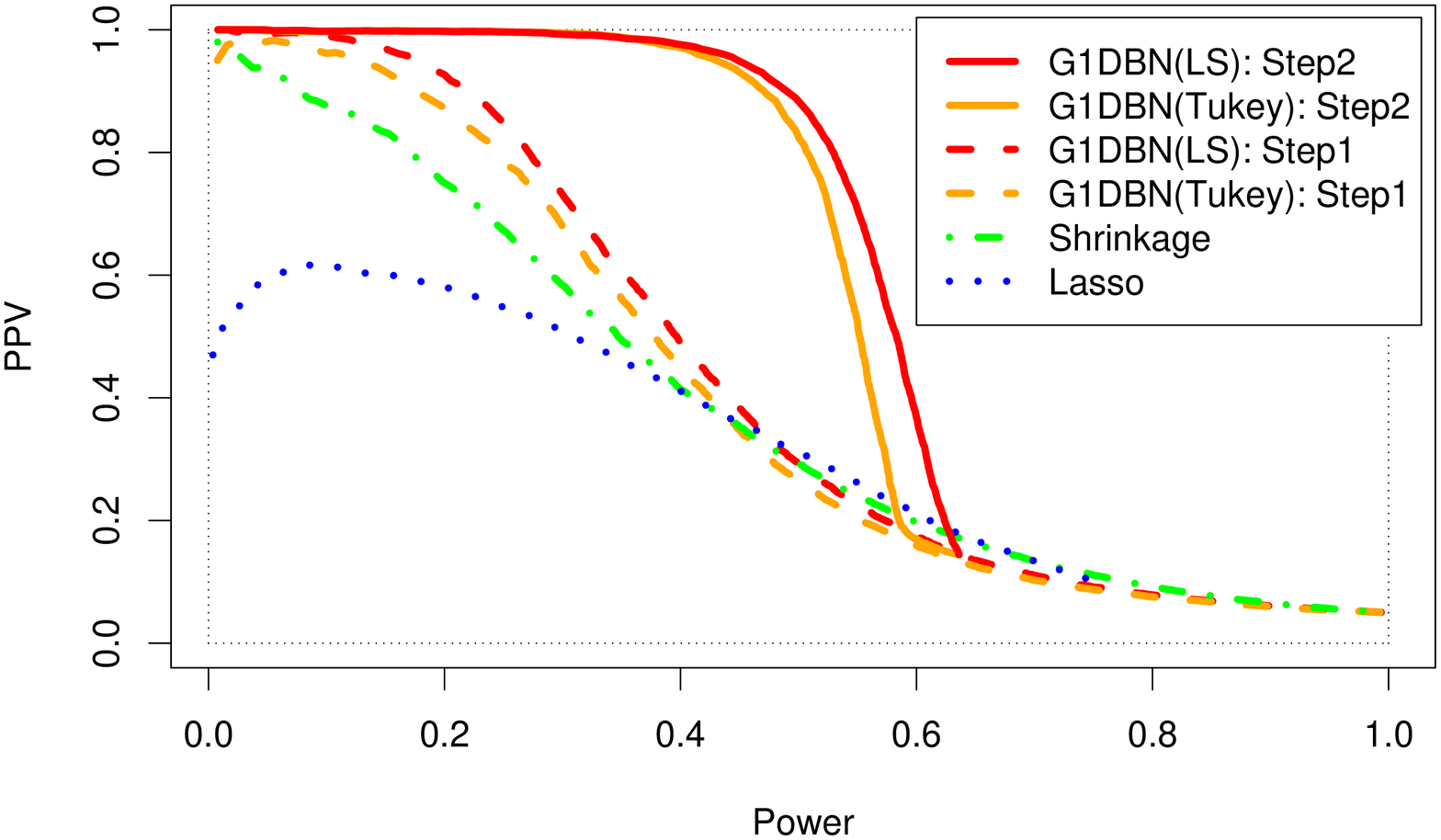}\\
\vspace{-3cm}
\hspace{-0.7cm}\includegraphics[width=12cm,height=12cm,angle=-90]{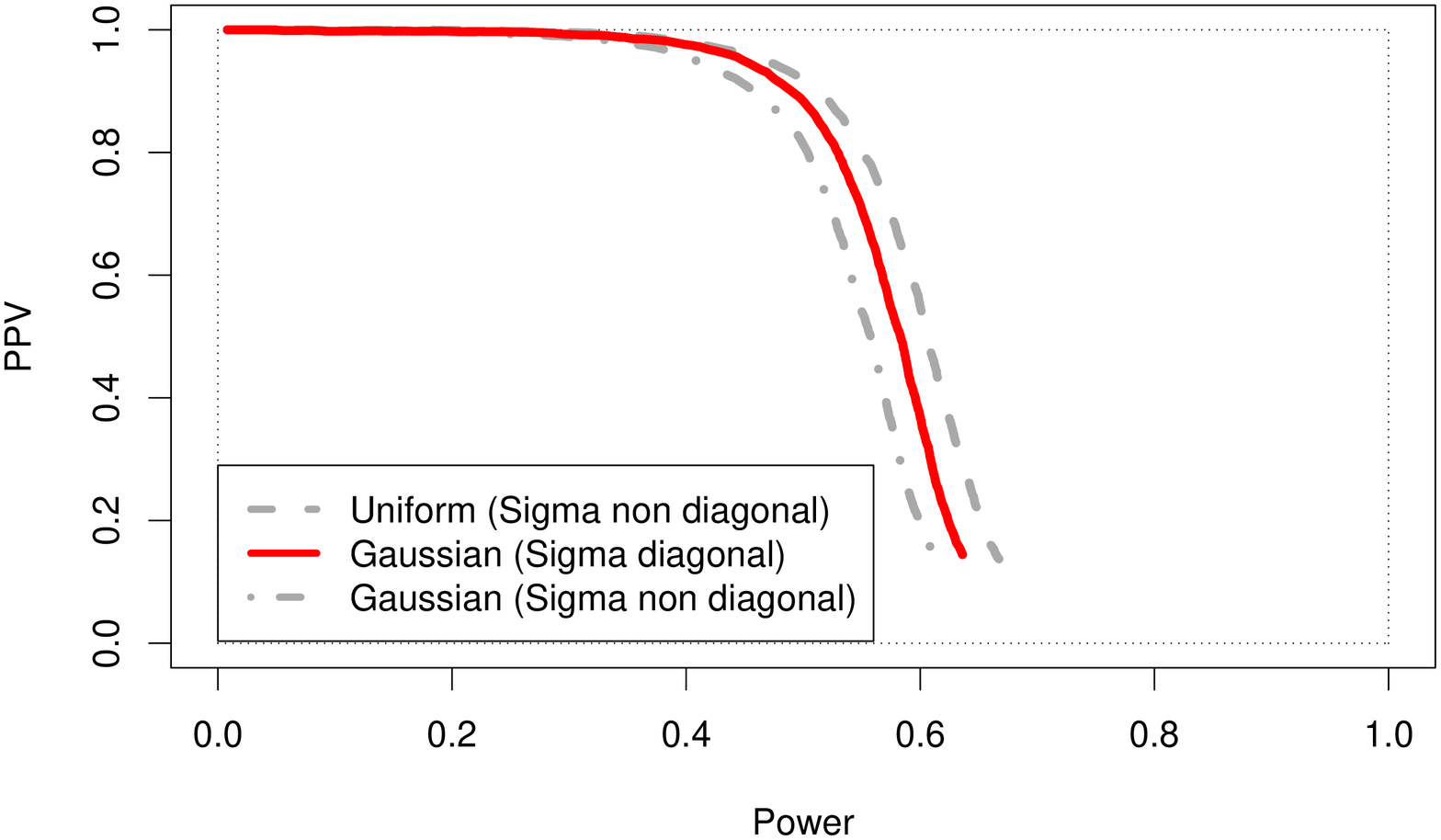}
\vspace{-1.5cm}

\caption{Precision-Recall (PR) curves obtained for network inference from simulated data ($n=20$). (A) Comparison of the inference procedures: G1DBN (LS or Tukey), shrinkage and lasso. Step 2 of the G1DBN approach drastically improves the results  (threshold $\alpha_1=0.7$). (B) Impact of  noisy data, simulated using a non diagonal matrix $\Sigma$ with either Gaussian or uniform noise, 
on  the G1DBN procedure (Step 2) computed  with LS estimates. } 
\label{fig:PRsimus}
\end{figure}


\newpage 
\paragraph{Simulation results}
 We show  on Figure \ref{fig:PRsimus} the results obtained with $n=20$, a length one can expect from existing gene expression   time series.
  Figure \ref{fig:PRsimus}A 
 displays the average Precision Recall  (PR) curves  obtained with the various inference approaches when the error  covariance matrix $\Sigma$ is diagonal and the noise distribution is Gaussian.
The Step~1 of the \textit{G1DBN} procedure computed  either with the LS estimator or with the Tukey estimator  (dashed lines)  
gives a very high PPV  for the very first selected edges. 
The Step 2 of  the \textit{G1DBN}  procedure (solid line) 
drastically improves the results. It allows to maintain the  PPV greater than 95~\% while the power goes up to 50\%. 
PR curves computed with the Huber  estimates (not shown)  
led to comparable results.
The lasso (dotted line) is  clearly outperformed by the other approaches and
  the shrinkage approach  (dashed-dotted line) gives results comparable to the Step~1 of the  \textit{G1DBN} procedure only.  The results of the three methods are naturally improved for greater values of $n$ but their relative perfomances are preserved (curves not shown).

We investigated the impact of the violation of the model assumptions. 
First we performed DBN inference on simulated data where the error  covariance matrix $\Sigma$ is   not diagonal (3\% of the coefficients outside the diagonal differ from 0)  and the noise distribution is  either Gaussian or uniform $\left (\mathcal{U}[-2;2] \right )$. As shown on Figure \ref{fig:PRsimus}B, the accuracy of the  \textit{G1DBN} procedure (Step 2) is not strongly affected 
when these assumptions on the noise distribution are not satisfied.
However, it is difficult to get rid of the $1^{st}$ order Markov Assumption which was chosen in order to reduce the model dimension.
When simulating an AR(2) model, the 2-order time dependencies existing in the model are missed. However, the 1-order time dependencies existing in the model are still recovered. Then, when considering a $2^{nd}$ order Markov process, an approximation can still be performed by successively inferring $1$- and $2$-order time dependencies.
Note that the procedure also performs well when the number of parents in the true
DAG  $\Gmin$ is greater than one (See 
 Supp. Material \cite{SuppMat}, Section A
 ).





\subsection{Analysis of microarray time course data sets}\label{ssec:realdata}

\begin{figure}
\begin{picture}(10,10)
\put(-5,10){\Large A}
\put(-5,-200){\Large B}
\end{picture}

\centering
\vspace{-2cm}

\includegraphics[height=9cm]{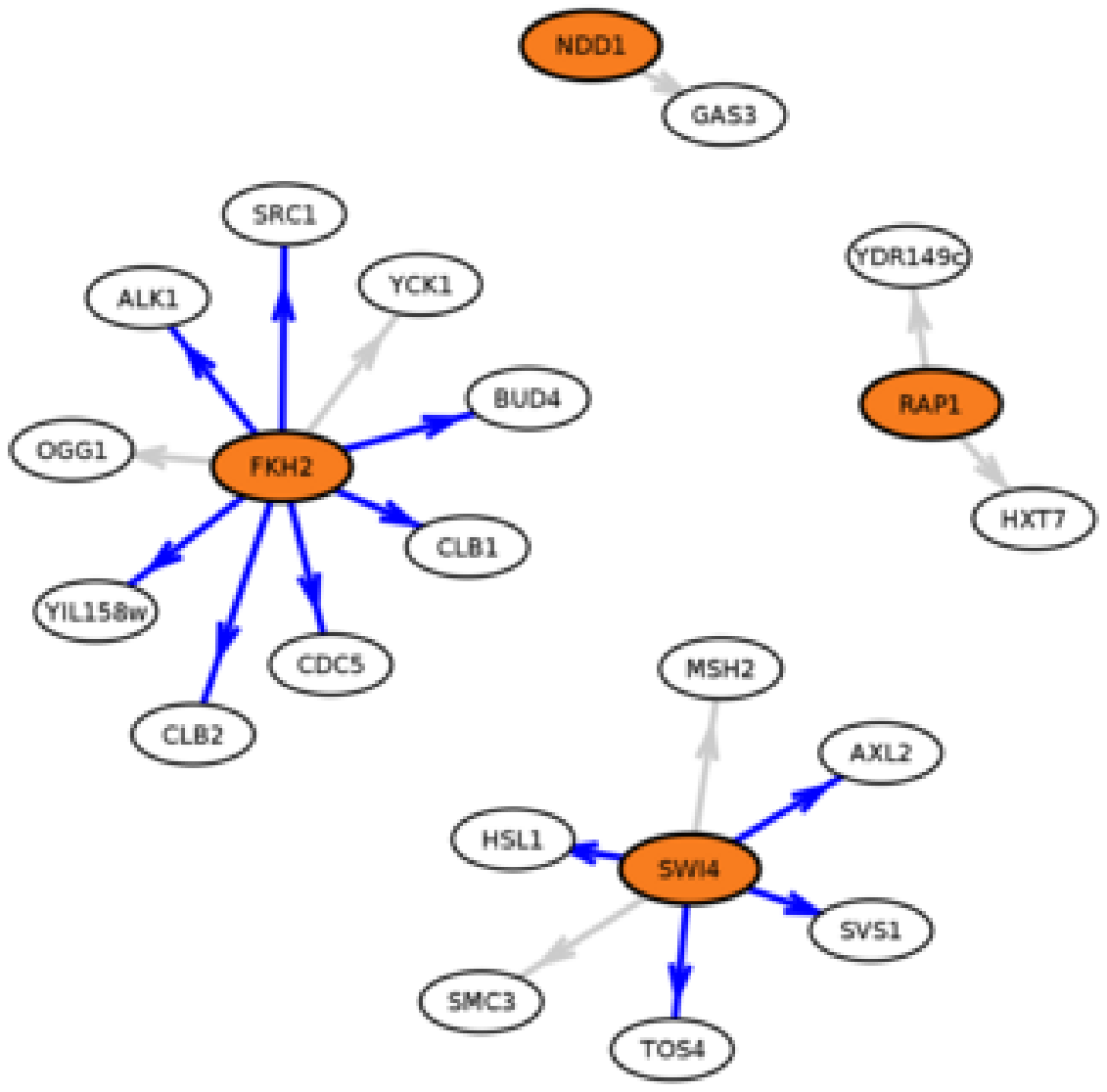} 

 \vspace{-1.5cm}

\includegraphics[height=10cm,width=10cm,angle=-90]{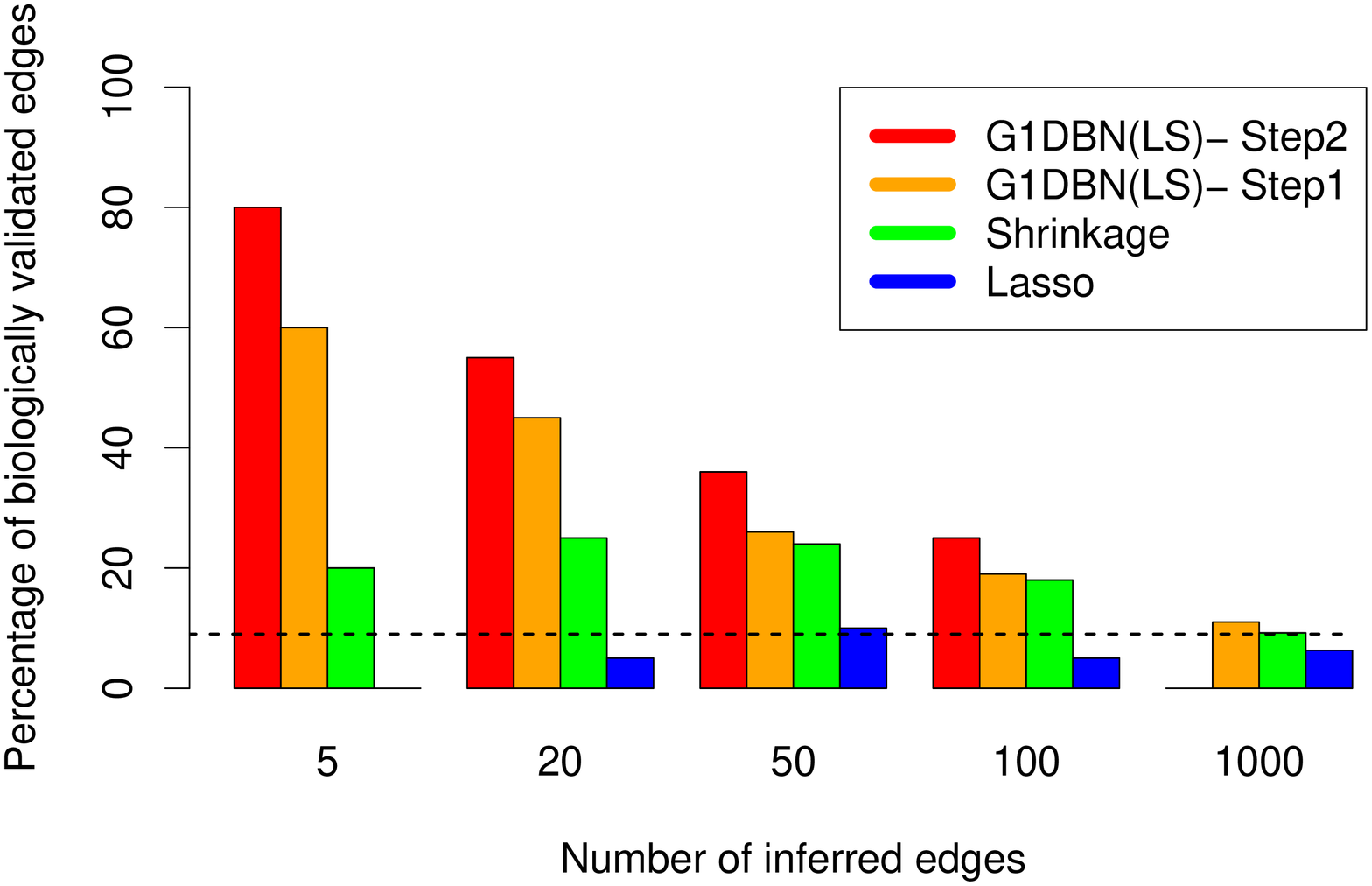}
\vspace{-0.2cm}

\caption{Some  results of the 18 TF-survey of \textit{S. cerevisiae} cell cycle.   (A)~DAG containing the 18 first selected edges with \textit{G1DBN} with LS estimates (PPV=60\%).
Colored nodes represent the TFs and the dark blue edges are validated by the Yeastract database. (B)~Percentage of validated edges out of the first 5 to 1000 edges  inferred with the \textit{G1DBN} procedure, after Step~2 or after Step 1 only, the shrinkage or the lasso procedure. The dashed line shows the proportion of validated edges out of the 786$\times$18 possible edges. }
\label{fig:res18TF}
\end{figure}

\paragraph{Spellman's Yeast cell cycle data set}\label{ssec:appliLevure}
We performed  dynamic network inference from the \textit{Saccharomyces  cerevisiae}  cell cycle data collected by Spellman et al.  \cite{Spellman98}. 
We used the $\alpha$~Factor-based synchronization data (18 time points) and we focus here on a set of 786 genes  which demonstrated  consistent  periodic  changes in transcription level (See Supplementary Material~\cite{SuppMat}, Section D.1
for more details).  
\newpage
We carried out two surveys on this dataset. First, we allow only a subset of 18 genes\footnote{The 18 genes code for proteins ACE2, FKH1, FKH2, GAT3, MBP1, MCM1, MIG2, NDD1, PHD1, RAP1, RME1, STB1, SUT1, SWI4, SWI5, SWI6, TEC1 and YOX1. consist of the overlap between the 786 genes under study and the 50 genes identified as putative TFs in a recent study by Tsai et al. \cite{Tsai05}.} identified as putative TFs to be possible parent genes (\ie to have edges pointing out towards other genes in DAG $\Gmin$) and look for their target genes.   Then we extend the search for parent genes to the whole dataset of 786 genes in a second survey. 
We set $\alpha_1$ threshold  for the  \textit{G1DBN} procedure  according to guidelines  detailed in Supplementary Material~\cite{SuppMat}, Section~B 
($\alpha_1=~0.1$ for the 18 TF-survey,  $\alpha_1=0.05$ for the  786 TF-survey). 

It is somehow difficult to assert the validity of the results obtained from real data as the whole regulatory machinery is not known yet. However the yeast cell cycle has been studied a lot and many regulation relationships have been recovered. 
We study the consistency of the first inferred edges with annotations in the 
  \textit{Yeastract}  database \cite{yeastract}, a curated repository currently listing 
 found regulatory associations  between TFs  and target genes in \textit{S. cerevisiae}.


In the 18 TF-survey,  the first few selected edges are biologically validated. In the DAG comprising the 18 first selected edges (Figure \ref{fig:res18TF}A), 11 edges refer to identified regulatory relationships (thick blue edges).
The first detected TFs are the genes coding for proteins FKH2, NDD1, RAP1 and SWI4. In particular, the proteins FKH2 (known as a TF with a major role in the expression of G2/M phase genes) and SWI4 (TF regulating late G1-specific transcription of targets) are pointed out as being essential TFs; they have the most target genes and the high majority ($73 \%$) of these regulatory relationships  is listed in Yeastract.

\begin{figure}
 \vspace{-6cm} 
 
 \hspace{-4.2cm}\includegraphics[width=23cm,height=27.8cm]{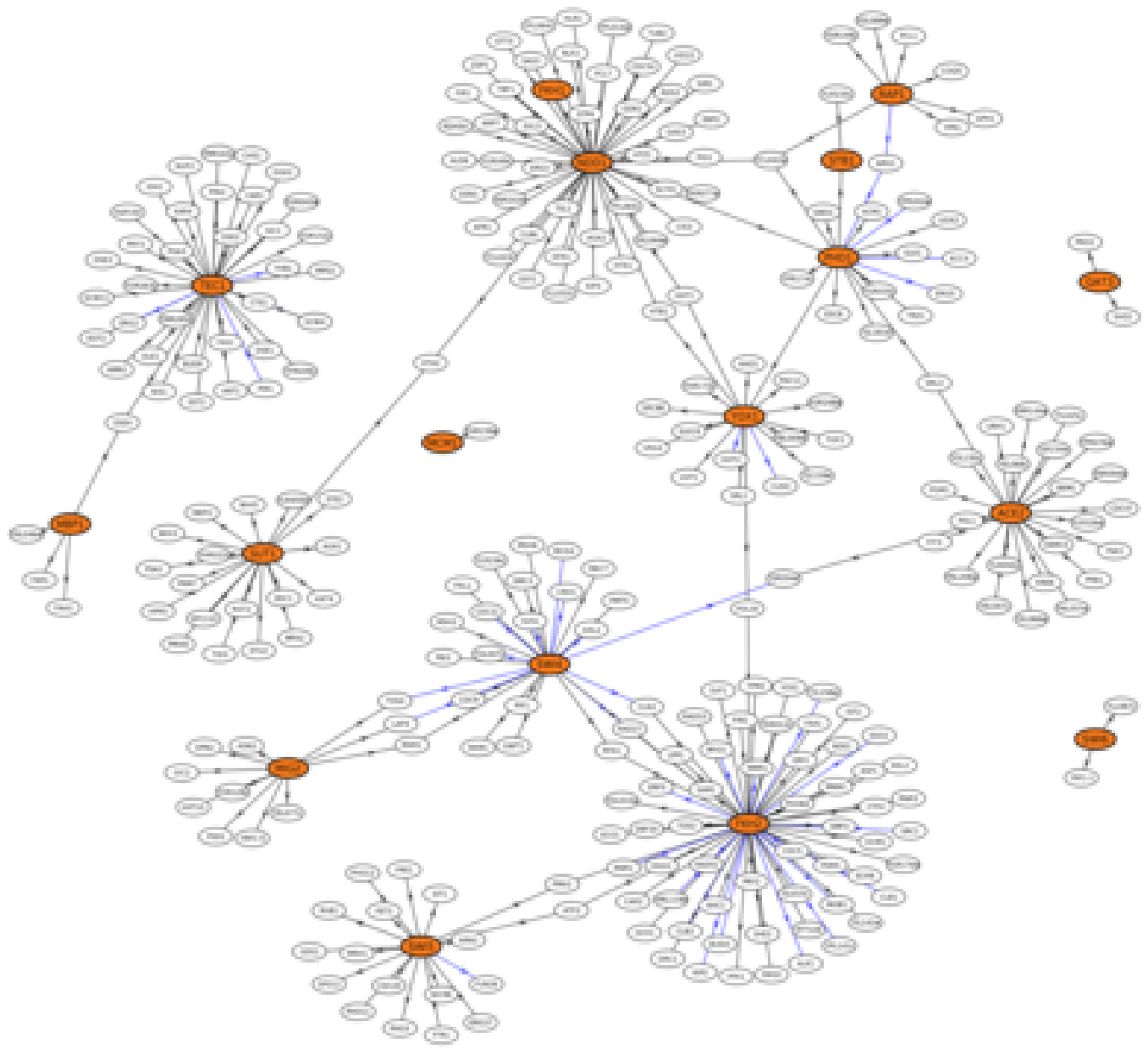} 
  \vspace{-1cm}
  
  \caption{\ DAG inferred by \textit{G1DBN} with LS estimates, using $\alpha_1=0.1$,\ \  $\alpha_2\!=\!0.0059$ (ensuring FDR$<0.01$), in the 18 TF-survey of the \textit{S. cerevisiae} cell cycle. The 17 colored nodes represent the 16 TFs selected as parent node out of the 18 TFs under study, plus node FKH1 which is selected as a target of NDD1. The dark blue edges are validated by Yeastract. This network contains $286$ genes and  $308$ edges. See the complete edges list in Supp. Material \cite{SuppMat}.
  }
  \label{fig:dagFDR18TF}
\end{figure}


As introduced in Section \ref{ssec:choiceThres}, we chose $\alpha_2$ threshold in order to keep the False Discovery Rate (FDR) smaller than 1\% with the approach by Benjamini and Hochberg \cite{Benjamini95}. This  lead to $\alpha_2=0.0059$.
The corresponding inferred DAG is shown in Figure \ref{fig:dagFDR18TF}.
The two proteins FKH2 and SWI4 are still part of the TFs having the most targets, together with NDD1, which is  an essential component of the activation of the expression of a set of late-S-phase-specific genes and TEC1, a transcription factor required for full Ty1 expression and Ty1-mediated gene activation (Ty transposable-element own for causing cell-type-dependent activation of adjacent-gene expression). 
The set of selected TFs is listed 
 in Supplementary Material \cite{SuppMat}, Section  D.2
, Table 1
, 
where the third column indicates the number of validated edges out of the selected ones. Except for NDD1, for which no target gene is listed in yeastract, one forth of the targets genes of the top four TFs are validated. 
\newpage
For a comparative overview, the histogram  of  Figure \ref{fig:res18TF}B displays the percentage of validated edges out of the first 5 to 1000 selected edges  inferred with each inference procedure
When considering the 1000 first inferred edges, the results are very  similar to what could be expected by chance only. Note that, as the Step 2 of G1DBN choose 308 edges only, it is not considered when comparing the 1000 first edges.


In the second survey including all the 786 genes as putative TFs, the dimension is far higher and the results are consequently more restricted. 
Indeed, 
the proportion of validated edges doesn't exceed $12.5 \%$, obtained  with the 2nd step of \textit{G1DBN} procedure among the first selected edges. However, this is still a subtantial result as compared with the proportion of validated edges (equal to $0.26\%$).
In order to keep the FDR smaller than 0.01, we chose $\alpha_2=0.0067$  by following the Benjamini and Hochberg approach \cite{Benjamini95}. The inferred DAG for the 786 TF-survey contains $437$ genes and $380$ edges. The display of this DAG, as well as the list of its edges and the list of the genes selected as TFs, is available in Supplementary Material \cite{SuppMat}.

\vspace{0.3cm}
\paragraph{Diurnal cycle on the starch metabolism of \em{A. Thaliana}} 
We  applied the  \textit{G1DBN} inference procedure to the expression time series data generated by Smith et al. \cite{Smith04} to investigate the impact of the diurnal cycle on the starch metabolism of \textit{Arabidopsis Thaliana}. We restricted our study to the 800 genes selected by Opgen-Rhein and Strimmer \cite{OpgenStrimmer07} as having periodic expression profiles\footnote{The data are available in the GeneNet R package  at                        \texttt{http://strimmerlab.org/} \texttt{software/genenet/html/ar}   \texttt{th800.html} or in our R package \textit{G1DBN} (arth800line).}.

\begin{figure}

 \vspace{-1.8cm}

\hspace{-5cm} \includegraphics[height=23.4cm]{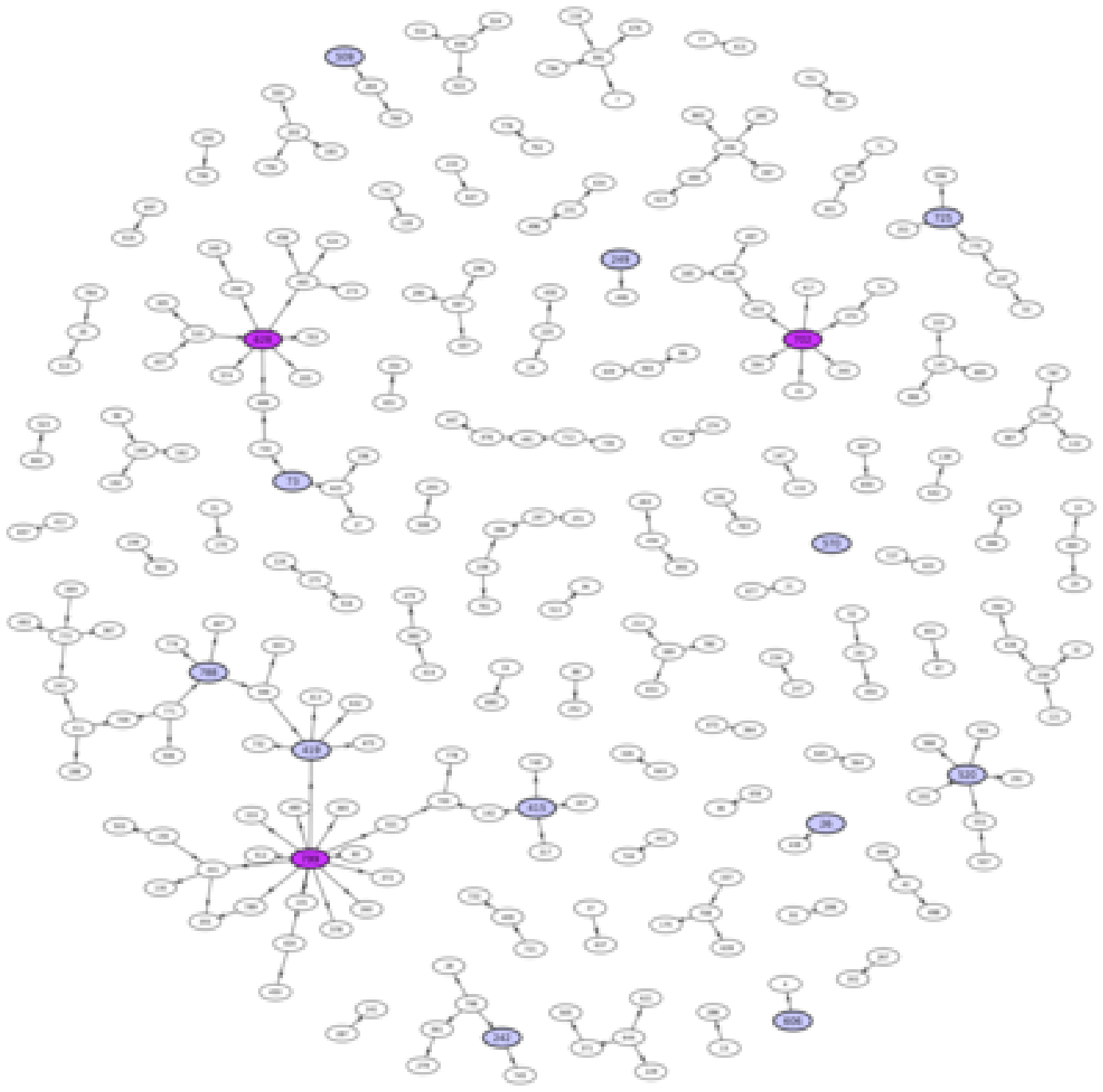} 
 \vspace{-1cm}

  \caption{DAG inferred  with \textit{G1DBN} from the data by Smith et al. $\!\!$ \cite{Smith04} in order to investigate starch metabolism of \textit{A. thaliana}  $\!\!$ (LS estimates, $ \alpha_1=0.1$, \newline $\alpha_2\!\!=\!\!0.005$ such that   FDR$<\!0.01\!$). $\!\!$ The dark colored nodes are the 3 nodes with the most targets, 2 out of them are known for being implicated in starch metabolism. The light colored nodes are parent nodes already identified as TF or DNA binding protein (See  Supp. Material\cite{SuppMat}, Section E, Table 2). This network contains  277 genes and 206 edges. See the edges list in Supp. Material.}
  \label{fig:dagarth}
\end{figure}

Using the heuristic approach detailed in Supplementary Material \cite{SuppMat}, Section B
, we choose  threshold $\alpha_1=0.02$  allowing the distribution of  the number of parents in the DAG $\gun$ having the number of 0-parent genes to dominate and the number of 1-parent genes to be half as large.
We set $\alpha_2=0.005$  in order to maintain the False Discovery Rate smaller than $0.01$ by using the approach by Benjamini and Hochberg \cite{Benjamini95} (See Section \ref{ssec:choiceThres} for details).  We recover the  DAG in Figure  \ref{fig:dagarth} which has  a ``hub'' connectivity structure. This network contains $206$ edges implicating $277$ different genes.
We may notice that this DAG differs from  the one inferred  by Opgen-Rhein and
Strimmer~\cite{OpgenStrimmer07}. However the edges selected by the three inference procedures discussed in this section differ somewhat  
 (See  the proportion of edges  in common by using the various inference approaches in Supplementary Material~\cite{SuppMat}, Section C
 ) and may, in fact, yield complementary information or insights.

\newpage

Among the `parent' nodes in the inferred DAG displayed in Figure \ref{fig:dagarth}, two nodes ($799$ and $628$) out of the three
having the most target refers to proteins that are  known to be implicated in 
starch metabolism.
Indeed,  node~$799$,  which has  $14$ `target' nodes,  refers  to DPE2
(DISPROPORTIONATING  ENZYME~2),  which is an essential  component of  the  pathway from
starch to sucrose and cellular metabolism in plant leaves at night. 
Node  $628$ ($6$~targets)  is a  transferase  (At5g24300)
implicated  in the starch  synthase. Node  $702$, which  is an  unknown protein
(At5g58220),  has  also  $6$  targets.   These  three  nodes  are
dark-colored  in the DAG  of Figure  \ref{fig:dagarth}.
Note that there is no prior knowledge regarding the role of each gene (TF or target) in this survey. As a consequence, some edges  might be inferred wrong way around\footnote{In particular if some assumption of the model is not satisfied. For instance if an essential TF is missing or if the regulation is not transcriptional, \ie does not depend on the amount of mRNA coding for the protein.}. Thus node~$799$, which is a gene coding for  an enzyme (DPE2), is most probably not a TF for its 14 apparent target genes.
However node $799$ is still the gene whose expression level best explains the expression of the 14 genes. Consequently these genes might be implicated in the same pathway as DPE2.
 The  remaining parent nodes have from $1$ to $4$ targets.  Among them, 9 genes, which are listed in Supplementary Material \cite{SuppMat}, Section E, Table 2
, have already been  identified as TFs
or as DNA binding proteins. 
These 9 nodes are light-colored in the displayed DAG. 
Finally a list of $37$ unknown proteins have been selected as parents in the inferred
DAG. Potentially implicated in the regulation machinery of starch metabolism,  these proteins represent a subset of genes which is relevant for further analyses.
See  more details on the inferred network displayed in Figure~\ref{fig:dagarth} 
in the Supplementary  Material \cite{SuppMat}. 


\section{Discussion and conclusion}

As more and more gene expression time series has become available, the need for efficient tools to analyze such data has become imperative.
In  this paper,  we first  determine sufficient  conditions for  Dynamic Bayesian Network modelling  of  gene  expression time
series. This type of modelling  offers a straightforward interpretation: the edges of the DAG $\Gmin$ defining the DBN exactly describe  the set of conditional dependencies
between  successive gene expression  levels. Having defined and characterized  low order
conditional dependence DAGs for DBNs, 
we point out relevant characteristics for the approximation of sparse DAGs.  In particular, under faithfulness assumption,  DAG $\Gmin$ is included in the $1^{st}$ order conditional dependence DAG $\gun$. 

From these  results, we develop  \textit{G1DBN},  a  novel procedure  for DBN inference, 
which makes  it possible  to tackle the  `small $n$,  large $p$'
estimation case that  occurs with genetic time series data. 
 Based  on the consideration of low  order conditional dependencies, the \textit{G1DBN} procedure proved to be powerful  on both simulated and real data analysis. 
With respect to other methods, the shrinkage  approach considerably improves  the  precision of  the overall
estimation of the partial correlation coefficients when the number of observations $n$ is small compared to the number of genes $p$.
 However, considering $1^{st}$ order conditional independence proved to be more efficient for DBN inference in terms of power  and PPV on simulated data, and gave promising results on real data analysis. 
As for the lasso,  one might notice that a drawback lies in the fact  that the edge selection
is done vertex by vertex whereas the DAG $\Gmin$ is globally sparse but not
uniformally.  As a consequence,  the lasso tends to  uniformally reduce
the number of  parents of each vertex instead of only  keeping the total
number of edges contained.

The power of the  \textit{G1DBN} procedure  comes  from  the accuracy improvement of the testing made possible by the dimension reduction.  Indeed, as the first step selection is based  on the $1^{st}$ order  conditional independence  consideration,  significance tests are performed in a model of dimension 4 (See Section~\ref{ssec:g1hat}). This represents a drastic dimension reduction compared to full order independence testing and makes the testing much more accurate.  
Thus, even if  there are more edges  in the DAG  $\gun$ than in the  true DAG $\Gmin$ (Proposition~\ref{prop:GminInGq}),  Step~1 of the procedure  is already very predictive. 

Throughout the analyses performed for this paper, we point out two major directions for further research.
On the one hand,  we noticed that the edges selected by the three inference procedures differ somewhat (See Supplementary Material \cite{SuppMat}, Section C). 
A further relevant study would consist of  analyzing in which way these DBN inference procedures could  have different strenghts and may be complementary. 
On the other hand, the use of  robust  estimators like Huber or Tukey bisquare did not allow a noticeable change of the inference approach on real data.
Another interesting survey lies in the investigation of which measures of dependence, like  non linear or other robust estimates, are the more pertinent  to  analyze gene expression data.


\newpage
\noindent \textbf{\Large{APPENDIX}}
\appendix

\section{Proofs}\label{app:proofs}

\subsection{Lemmas \ref{lem:Gfull}  to  \ref{lem:ijIndep}  and proofs}\label{sapp:Lemmas}
\begin{Lem}\label{lem:Gfull}
  Under  Assumptions   \ref{ass:m1}  and  \ref{ass:cindep},   the  probability
  distribution $\proba$ admits  a DBN representation according to  a DAG whose
  edges only join nodes representing variables observed at two successive time points, at least according to
  DAG $\Gfull = ( X , \{(X^j_{t-1}, X^i_t)\}_{i,j \in \p, t>1} )$ which has edges
  between any pair of successive variables.
\end{Lem}

\begin{proof}[\textbf{Proof of Lemma \ref{lem:Gfull}}]
From assumption \ref{ass:m1}, the density $f$
of the joint probability distribution of process~$X$  be written as the product of conditional densities,
\begin{equation}\label{eq:m1}
f(X)= f(X_1) \ \prod_{t =2}^n f (X_t \vert X_{t-1}),
\end{equation}

\noindent where $  f (X_t \vert X_{t-1})$ refers  to the  density
of the conditional probability distribution 
of~$X_t$ given $X_{t-1}$.

From Assumption \ref{ass:cindep}, for all $t>1$, the conditional density
$f (X_t  \vert X_{t-1})$ can be written as  the product of the  conditional density of
each variable  $X^i_t$ given  the set of  variables $X_{t-1}$ observed  at the
previous time,
\begin{equation}\label{eq:fact}
 f (X_t \vert X_{t-1}) = \prod_ {i \in \p} f (X^i_t \vert X_{t-1}). 
\end{equation}


From equations (\ref{eq:m1}) and (\ref{eq:fact}), the density $f$ writes as the product of the conditional density of
each variable $X^i_t$ given its parents in $\Gfull$. 
From Proposition \ref{Prop:BNrep}, the probability distribution $\proba$ admits a BN representation
according to $\Gfull$. 
\end{proof}\\

\begin{Lem}\label{lem:interGraph}
Assume  the joint probability  distribution $\proba$ of
process $X$ has density~$f$ with respect to Lebesgue measure on $\R^{p \times n}$. 
If $\proba$ factorizes according to two
different subgraphs of $\Gfull$, $\G_1$ and $\G_2$, then $\proba$ factorizes according to $\G_1 \cap \G_2$. 
\end{Lem} 

\textit{\textbf{From  Lemma \ref{lem:interGraph}, it is straightforward that, among the  DAGs included in $\Gfull$, there exists   a minimal DAG  (denoted by $\Gmin$  in the paper) according to which the probability distribution $\proba$ factorizes, thus establishing a BN~representation of process~$X$.}}

\paragraph{}
\begin{proof}[\textbf{Proof of Lemma \ref{lem:interGraph}}\label{app:interGraph}]
Consider a discrete-time stochastic process 
$X \!=\! \{X^i_t;\ i \in~\p,$ $  t \in \n \}$
whose joint probability $\proba$ distribution has the density $f$ with respect to
Lebesgue measure on $\R^{p \times n}$. 

Let $\G_1$ and $\G_2$ be two different subgraphs of $\Gfull$ according to which
the joint probability distribution $\proba$ factorizes. 
Let $i \in \p$, $t \in \n$, we consider the random variable $X^i_t$. 

We denote as follows,
\vspace{-0.3cm}

\begin{itemize}
\item the following subsets of $\p$,

$pa_1 = \{j \in \p; X^j_{t-1} \in pa(X^i_t, \G_1) \}$

$\overline{pa}_1 = \p \bs \{pa_1\}$

$pa_2 = \{j \in \p; X^j_{t-1} \in pa(X^i_t, \G_2)\}$

$\overline{pa}_2 = \p \bs \{pa_2\}$

\item and the densities of  the joint or marginal probability distributions of
  $(X^i_t, X_{t-1})$,

$g : \R^{p+1} \rightarrow \R$  the density of the joint probability distribution of $(X^i_t, X_{t-1})$,

$g^i$ the density of the probability distribution of $X^i_t$,

$g^{\p}$ the density of the joint probability distribution of $(X_{t-1})$,

$g^{i,pa_1}$  the density of  the joint  probability distribution  of $(X^i_t,
X^{pa_1}_{t-1})$ where,

\hspace{1cm} $X^{pa_1}_{t-1}=pa(X^i_t, \G_1)$,

$g^{i,\overline{pa}_2}$ the density of the joint
probability   distribution   of  $(X^i_t,   X^{\overline{pa}_2}_{t-1}) $ where

\hspace{1cm} $ X^{\overline{pa}_2}_{t-1}= X_{t-1} \bs \{ pa(X^i_t,\G_2) \}) $,

etc... 








\end{itemize}
\vspace{-0.3cm}

In         the        following,         $y         \in        \R$,
$x=(x_1, ..., x_p) \in \R^p$ and we denote by $x_{pa_1}= \{x_j;j \in~pa_1\} \in
\R^{\vert pa_1 \vert}$ 
(Thus $x=(x_{pa_1},x_{\overline{pa}_1})=(x_{pa_2},x_{\overline{pa}_2}) \in \R^p$). 
As the probability distribution $\proba$ factorizes according to $\G_1$, we derive from the DAG
theory the conditional independence,
\vspace{-0.2cm}

$$X^i_t \indep X^{\overline{pa}_1}_{t-1} \vert X^{pa_1}_{t-1},$$ 

\noindent that is,
\vspace{-0.3cm}

$$\forall y \in \R, \forall x \in \R^p, \ \ \frac{g(y,x)}{g^{\p}(x)} =\frac{g^{i,pa_1}(y,x_{pa_1})}{g^{pa_1}(x_{pa_1})}.$$


Equivalent results derived from the factorization according to $\G_2$ gives,
\vspace{-0.2cm}

$$\forall y \in \R, x \in \R^p, \n \ g^{i,pa_2}(y,x_{pa_2}) = \frac{g^{i,pa_1}(y,x_{pa_1})}{g^{pa_1}(x_{pa_1})} g^{pa_2}(x_{pa_2}).$$
By taking  the integral  with respect to  $x_{pa_2 \cap  \overline{pa}_1}$, we
write for all $y \in \R$, for all $x_{pa_1 \cup pa_2} \in \R^{\vert pa_1 \cup pa_2 \vert }$,
\vspace{-0.3cm}

\begin{eqnarray*}
\int g^{i,pa_2}(y,x_{pa_2}) d(x_{pa_2 \cap \overline{pa}_1})
&=& \int \frac{g^{i,pa_1}(y,x_{pa_1})}{g^{pa_1}(x_{pa_1})} g^{pa_2}(x_{pa_2})
d(x_{pa_2 \cap \overline{pa}_1})\\
g^{i,pa_1 \cap pa_2 }(y,x_{pa_1 \cap pa_2}) &=& \frac{g^{i,pa_1}(y,x_{pa_1})}{g^{pa_1}(x_{pa_1})} g^{pa_1 \cap
pa_2}(x_{pa_1 \cap pa_2})
\end{eqnarray*} 

Finally we have,
$$\forall y \in \R, \forall x \in \R^p, \ \ \frac{g(y,x)}{g^{\p}(x)}=\frac{g^{i,pa_1 \cap pa_2}(y,x_{pa_1 \cap
pa_2})}{g^{pa_1 \cap pa_2 }(x_{pa_1 \cap pa_2})},$$
\noindent that is the conditional density of the probability distribution of $X^i_t$
given $X_{t-1}$ is the conditional  density of the  probability distribution  of $X^i_t$
given  $X^{pa_1  \cap   pa_2}_{t-1}$.  Then  $\proba$  factorizes
according to $\G_1 \cap \G_2$. 
\end{proof}

\begin{Lem}\label{lem:ijIndep}\textbf{ (Conditional independence between non
    adjacent  successive variables) } Let $\G$ be a subgraph of $\Gfull$
  according  to  which  the  probability  distribution $\proba$  admits  a  BN
  representation.   For any pair  of successive  variables $(X^j_{t-1},X^i_t)$
  which are non adjacent in $\G$, we have
$$X^i_t \indep X^j_{t-1}\ \vert \ pa(X^i_t,\G) \ \text{ and }\ 
 X^i_t \indep X^j_{t-1}\ \vert \ pa(X^i_t,\G) \cup S,$$
for all $S$ subset of $\{X^k_u; k \in \p, u <t\}$. 
\end{Lem}

As an illustration of Lemma \ref{lem:ijIndep}, assume $\proba$
admits  a BN  representation according  to  the DAG  of Figure  \ref{fig:motifDBN}C. 
There is no edge between $X^3_t$ and $X^1_{t+1}$ in this DAG.  
Now consider in Figure \ref{fig:moralG}B the moral graph of the smallest ancestral graph containing $X^3_t$,
$X^1_{t+1}$ and the parents $(X^1_t,X^2_t)$ of $X^1_{t+1}$. 
The set $(X^1_t,X^2_t)$ 
  blocks all paths between $X^3_t$ and $X^1_{t+1}$. From Proposition \ref{prop:gAnc}, we have $X^1_{t+1}
\indep X^3_{t}\ \vert \ pa(X^1_{t+1},\G)$. 

\paragraph{}
\begin{proof}[\textbf{Proof of Lemma \ref{lem:ijIndep}}] Assume $\proba$ admits a BN representation according to $\G$, a
subgraph  of  $\Gfull$.  Let  $X^j_{t-1}$  and $X^i_t$  be  two  \textit{non
    adjacent}  vertices of  $\G$ (there  is no  edge between  them in  $\G$) and
  consider   the   moral  graph   $(\G_{An   \left  (X^i_t\cup   X^j_{t-1}\cup
      pa(X^i_t,\G) \right )} )^m$ of the smallest ancestral set containing the
  variables $ X^i_t$,  $X^j_{t-1} $ and the parents  $pa(X^i_t,\G)$ of $X^i_t$
  in $\G$.
As DAG $\G$ is a subgraph of $\Gfull$, 
the set of parents
$pa(X^i_t,\G)$ blocks all paths between $X^j_{t-1}$ and $X^i_t$ in the moral
graph $(\G_{An \left (X^i_t\cup X^j_{t-1}\cup pa(X^i_t,\G) \right )} )^m$. 
From Proposition \ref{prop:gAnc}, 
this establishes the conditional
independence $X^i_t \indep X^j_{t-1}\ \vert \ pa(X^i_t,\G)$.

This result  holds for the conditioning  according to any  subset $S \subseteq
\{X^k_u; k \in \p, u <t\}$. 
\end{proof}


\subsection{Proof of Propositions \ref{prop:Gmin}, \ref{prop:G1},  \ref{prop:GminInGq} and  \ref{prop:NpaGq}}\label{sapp:proofProps}

\begin{proof}[\textbf{Proof of Proposition \ref{prop:Gmin}}]
First, we show that $\proba$ admits
a BN representation according to $\Gmin$. 
Let $i,j \in \p$ such that $ X^i_t \indep X^j_{t-1} \vert
X^{\pj}_{t-1}$, then we have, 
\vspace{-0.3cm}

$$f(X^i_t \vert X_{t-1})=f(X^i_t \vert X^{\pj}_{t-1}).$$ 
\noindent Under Assumptions \ref{ass:m1} and \ref{ass:cindep}, from Lemma
\ref{lem:Gfull} (See Appendix \ref{sapp:Lemmas}) and Prop. \ref{Prop:BNrep}, $\mathbb{P}$ admits
a BN representation according to the DAG $(X, E(\Gfull) \ \bs \ ( X^j_{t-1},
X^i_t))$ which has the edges of $\Gfull$ except for the edge $( X^j_{t-1},
X^i_t)$.
This holds for any pair of successive variables that are conditionally independent.

From Lemma \ref{lem:interGraph} (See Appendix \ref{sapp:Lemmas}), $\mathbb{P}$ admits
a BN representation according to the intersection of the DAG  $(X, E(\Gfull) \ \bs \ ( X^j_{t-1},
X^i_t))$ for any pair $( X^i_t, X^j_{t-1})$ such that $ X^i_t \indep X^j_{t-1} \vert
X^{\pj}_{t-1}$, that is DAG $\Gmin$. 


Also, DAG $\Gmin$ cannot be reduced. 
$\ $ Indeed, let $(X^l_{t-1},X^k_t)$ be an edge of $\Gmin$ \newline 
and assume that $\proba$ admits a BN representation according to $\Gmin \bs
 (X^l_{t-1},X^k_t)$,
that is DAG $\Gmin$ with the edge $(X^l_{t-1},X^k_t)$ removed. 
From Lemma \ref{lem:ijIndep}  (Appendix~\ref{sapp:Lemmas}), we have $ X^k_t \indep X^l_{t-1} \vert
X^{\p_{l}}_{t-1}$, which contradicts $(X^l_{t-1},X^k_t) \in V(\Gmin)$
(\ie $X^k_t \not\indep X^l_{t-1} \vert X^{\p_{l}}_{t-1}$).  

\end{proof}




\paragraph{}
\begin{proof}[{\textbf{Proof of Proposition \ref{prop:G1}}}]

First, from Corollary \ref{cor:GqInGmin}, $ \Gmin \supseteq \gun$. 

Second, let $X$ be a Gaussian process and $(X^j_{t-1},X^i_t) \in E(\Gmin)$, then
  according to Proposition \ref{prop:Gmin}, $X^i_t \not \indep X^j_{t-1} \ \vert \ X^{\pj}_{t-1}$.  $\ \ $ Since $X$ is
Gaussian, this implies

\noindent $Cov (X^i_t,X^j_{t-1} \vert X^{\pj}_{t-1} ) \not= 0.$   

Now assume that there exists $k\not=j$, such that $X^i_t \indep X^j_{t-1}\ \vert \
X^k_{t-1} $
ie $(X^j_{t-1},X^i_t)
\notin E(\gun)$.
We are going to prove that this contradicts the nullity of covariance  $Cov (X^i_t,X^j_{t-1} \vert X^{\pj}_{t-1} )\not=0$. 

Let  $l$  be an  element  of $\p  \bs  \{j,k\}$.  
 The conditional  covariance
 $Cov(ij\vert k,l) = Cov(X^i_t ,X^j_{t-1} \ \vert X^k_{t-1} , X^l_{t-1})$ can be written,



\begin{eqnarray*}
Cov(ij\vert k,l) &\!=\! &Cov(X^i_t ,X^j_{t-1} \ \vert X^k_{t-1})
 \!-\! \frac{ Cov(X^i_t ,X^l_{t-1} \ \vert X^k_{t-1}) Cov(X^j_{t-1},X^l_{t-1}
  \vert X^k_{t-1})}{Var(X^l_{t-1} \vert X^k_{t-1})},\\
&\! =\!& Cov(X^i_t  ,X^j_{t-1} \
\vert X^k_{t-1}) \times \left[ 1-\frac{(Cov(X^j_{t-1},X^l_{t-1} \vert X^k_{t-1}))^2}{Var(X^j_{t-1}      \vert
    X^k_{t-1})Var(X^l_{t-1} \vert X^k_{t-1} )} \right]
\end{eqnarray*}
$$\hspace{5cm} - \ \frac{Cov(X^j_{t-1},X^l_{t-1}\vert X^k_{t-1}) Cov(X^i_t
      ,X^l_{t-1} \ \vert X^k_{t-1},X^j_{t-1})}{Var(X^l_{t-1} \vert X^k_{t-1})}.$$
However both terms in the latter expression of $Cov(ij\vert k,l) $ are null:
\begin{itemize}
\item since $X^i_t \indep X^j_{t-1}\ \vert \ X^k_{t-1} $,
then $Cov(X^i_t ,X^j_{t-1} \ \vert X^k_{t-1} )=0$,


\item as $\NpaMaxGmin \leq 1$, $X^j_{t-1} $  is the only parent of $X^i_t $ in
  $\Gmin$.    So   the   variable   $X^j_{t-1}$   and  thus   also   the   set
  $(X^j_{t-1},X^k_{t-1})$ blocks all paths between $X^l_{t-1}$ and $X^i_t $ in
  the  moral  graph of  the  smallest  ancestral  set containing  $X^i_t  \cup
  X^{j,k,l}_{t-1}$.  Then we  have,  $X^i_t \indep  X^l_{t-1}\  \vert \  \left
    \{X^j_{t-1},X^k_{t-1}\right \}$,  that is  $ Cov(X^i_t ,X^l_{t-1}  \ \vert
  X^k_{t-1},X^j_{t-1}) =0$.
\end{itemize}

Then $Cov(ij\vert k,l) =0$. 
By  induction,   we  obtain $Cov (X^i_t,X^j_{t-1} \vert X^{\pj}_{t-1} )=0$   leading  to  a
contradiction    with    $(X^j_{t-1},X^i_t)    \in    E(\Gmin)$.    Therefore
$(X^j_{t-1},X^i_t) \in \gun$ and  we have $ \Gmin \subseteq~\gun$.

\end{proof}

\begin{proof}[{\textbf{Proof of Prop \ref{prop:GminInGq} }}]

Let $(X^j_{t-1} ,X^i_t)  \in E(\Gmin)$.  Assume that $(X^j_{t-1}
  ,X^i_t) \notin E(\gq)$ then there exists a subset of $q$ variables $X^Q_{t-1}$
  with   respect  to   which   $X^j_{t-1}$  and   $X^i_t$  are   conditionally
  independent. 
  From faithfulness, the subset  $X^Q_{t-1}$ separates $X^j_{t-1}$ and $X^i_t$
  in  the moral graph  of the  smallest ancestral  set containing  $X^i_t \cup
  X^j_{t-1}  \cup  X^Q_{t-1}$.  This  contradicts  the presence  of  the  edge
  $(X^j_{t-1} ,X^i_t)$ in $\Gmin$. 
\end{proof}

\paragraph{}
\begin{proof}[{\textbf{Proof of Prop \ref{prop:NpaGq} }}]

From faithfulness, $\Gmin \subseteq \gq$.  Then for all $i$ in $\p$, for all $t >1$, we have
  $  N_{pa}(X^i_t,\Gmin)  \  \leq  \  N_{pa}(X^i_t,\gq) \  \leq  ~q$.
  
     From Proposition   \ref{prop:infq},    $(X^j_{t-1},X^i_t)   \notin   E(\Gmin)   \
  \Rightarrow \ (X^j_{t-1},X^i_t) \notin  E(\gq) $, that is 
$ (X^j_{t-1},X^i_t) \in~E(\gq) \ \Rightarrow \ (X^j_{t-1},X^i_t) \in~E(\Gmin) $.  

\end{proof}

\vspace{0.3cm}

\bibliographystyle{plain}
\bibliography{biblioDBN}

\end{document}